\documentclass[10pt,a4paper]{article}
\usepackage{tikz}
\usepackage[width=14.00cm, height=25.00cm]{geometry}
\usepackage[T1]{fontenc}
\usepackage[utf8]{inputenc}
\usepackage{amsmath,amssymb,amsfonts,amsthm,amscd, bm, mathtools}
\usepackage{enumerate}
\usepackage{todonotes}
\usetikzlibrary{positioning}
\usetikzlibrary{decorations.pathmorphing}
\usetikzlibrary{decorations.text}
\newtheorem{thm}{Theorem}[section]
\newtheorem{prop}[thm]{Proposition}
\newtheorem{lemma}[thm]{Lemma}
\newtheorem{cor}[thm]{Corollary}
\newtheorem{defn}[thm]{Definition}

\newtheorem{conj}[thm]{Conjecture}
\newtheorem{ex}[thm]{Example}

\title{On cycle covers of infinite bipartite graphs}
\author{Leandro Aurichi, Paulo Magalhães Júnior and Lyubomyr Zdomskyy  }

\newcommand{\Addresses}{{
  \bigskip
  \footnotesize
  L. ~Aurichi, \textsc{Instituto de Ci\^encias Matem\'aticas e de Computa\c c\~ao, Universidade de S\~ao Paulo\\
	Avenida Trabalhador s\~ao-carlense, 400,  S\~ao Carlos, SP, 13566-590, Brazil}\par\nopagebreak
  \textit{E-mail address}, L.~Aurichi: \texttt{aurichi@icmc.usp.br}

  \medskip
  P.~Magalhães Jr, \textsc{Instituto de Ci\^encias Matem\'aticas e de Computa\c c\~ao, Universidade de S\~ao Paulo\\
	Avenida Trabalhador s\~ao-carlense, 400,  S\~ao Carlos, SP, 13566-590, Brazil}\par\nopagebreak
  \textit{E-mail address}, P.~Magalhães Jr: \texttt{pjr.mat@usp.br}
  
  \medskip
  L.~Zdomskyy, \textsc{Institut für Diskrete Mathematik und Geometrie, Technische Universität Wien\\
	Wiedner Hauptstrasse 8-10/104, 1040 Wien, Austria}\par\nopagebreak
  \textit{E-mail address}, L.~Zdomskyy: \texttt{lzdomskyy@gmail.com}
 
}}

\date{}

\begin{document}
\maketitle

\begin{abstract}
Given a graph $G$ and a subset $X$ of vertices of $G$ with size at least two, we denote by $N^2_G(X)$ the set of vertices of $G$ that have at least two neighbors in $X$. We say that a bipartite graph $G$ with sides $A$ and $B$ satisfies the double Hall property if for every subset $X$ of vertices of $A$ with size at least 2, $\vert N^2_G(X)\vert \geq \vert X\vert$. Salia conjectured that if $G$ is a bipartite graph that satisfies the double Hall property, then there exists a cycle in $G$ that covers all vertices of $A$. In this work, we study this conjecture restricted to infinite graphs. For this, we use the definition of ends and infinite cycles. It is simple to see that Salia's conjecture is false for infinite graphs in general. Consequently, all our results are partial. Under certain hypothesis it is possible to obtain a collection of pairwise disjoint 2-regular subgraphs that covers $A$. We show that if side $B$ is locally finite and side $A$ is countable, then the conjecture is true. Furthermore, assuming the conjecture holds for finite graphs, we show that it holds for infinite graphs with a restriction on the degree of the vertices of $B$. This result is inspired by the result obtained by Barát, Grzesik,  Jung, Nagy and Pálvölgyi for finite graphs. Finally, we also show that if Salia's conjecture holds for some cases of infinite graphs, then the conjecture about finite graphs presented by Lavrov and Vandenbussche is true.
\end{abstract}

\section{Introduction:}
\paragraph{}

A widely studied problem in graph theory is finding conditions under which a graph or a subset of it can be covered by a cycle, a path, a family of cycles or a family of paths. In this sense, there are several conjectures suggesting sufficient hypotheses to guarantee the existence of such covers, for example the double cycle cover conjecture (\cite{Genest},\cite{surveycdcc}), the faithful cycle cover conjecture (\cite{F-limit},\cite{cyclecocycle}) and Salia's conjecture. In this work, we will discuss the latter. In this article, we will use the definitions and notation as in \cite{diestelbook}.

Let $G(A,B)$ be a bipartite graph with sides $A$ and $B$. We denote by $N_{G}(X)$ the set of neighbors of $X\subset A$. We say that a matching $M\subset E(G)$ is an $A$-perfect matching if for every $v\in A$, $\vert E(v)\cap M\vert =1$, where $E(v)$ is the set of edges in $E(G)$ incident to $v$. One of the best-known theorems in graph theory is Hall's Theorem which states that a bipartite graph $G(A,B)$ has an $A$-perfect matching if and only if, for every $X\subset A$, $\vert N_G(X)\vert \geq \vert X\vert$. Hall's Theorem, while generally false for infinite graphs, remains valid for infinite bipartite graphs where all vertices in set $A$ have finite degree, as demonstrated by Philip Hall in \cite{HALL}. Given $X\subset V(G)$ with $\vert X\vert \geq 2$, we define $N^2_{G}(X)$ as the set of vertices that have at least two neighbors within the set $X$. Since a cycle in $G$ that covers all vertices of $A$ can be viewed as the union of two $A$-perfect matches that is connected, then it is natural to want to somehow duplicate the Hall property to try to obtain a cycle in $G$ that covers all vertices of $A$. For this purpose the double Hall property is defined as follows: 

\begin{defn}
Let $G(A,B)$ be a bipartite graph with sides $A$ and $B$. If $\vert A\vert \geq 2$ and $\vert N^2(X)\vert \geq \vert X\vert$ for all $X\subset A$ of size at least 2, then $G$ satisfies the \textbf{double Hall property}. Briefly, we will denote the graphs with this property by \textbf{dHp}.    
\end{defn}

If a graph satisfies the double Hall property we will call it a dHp graph. In his doctoral thesis \cite{Salia}, Salia presented the following conjecture:   

\begin{conj} [Salia's conjecture \cite{Salia}]\label{con1}
    Let $G$ be a bipartite graph with sides $A$ and $B$ such that $\vert N^2(X)\vert \geq \vert X\vert$ for all $X\subset A$ of size at least 2. For every $X\subset A$, $\vert X\vert\geq 2$, there is a cycle $C_X$ in $G$ such that $V(C_X)\cap A=X$.
\end{conj}

Salia's conjecture is inspired by a similar conjecture presented in \cite{zbMATH07303525}. Note that in order to prove the conjecture it is sufficient to prove that any dHp graph $G(A,B)$ contains a cycle covering $A$.

\begin{defn}
    We say that a dHp graph $G(A,B)$ satisfies the \textbf{weak Salia's Conjecture}, if there exists a family of mutually disjoint cycles covering $A$.
\end{defn}

Example \ref{thm3} shows that countable dHp graphs do not even have to satisfy the perfect matching, no word about weak Salia's Conjecture. In this paper we shall present several sufficient conditions for an infinite graph to satisfy the (weak) Salia's Conjecture.

In \cite{barát2023double}, it is proved that for finite graphs it is possible to cover the set $A$ with a 2 -factor and that for graphs under certain restrictions on the degree of the vertices of $B$, Salia's conjecture holds. In this work, we will study Salia's conjecture restricted to infinite graphs. We will need to recall the concepts of infinite cycle and end of a graph.

In a graph $G$, a \textbf{ray} is an one-way infinite path, whose infinite connected subgraphs are called its \textbf{tails}. Similarly, a two-way infinite path is referred to as a \textbf{double-ray}. Over the set $\mathcal{R}(G)$ of rays of $G$, one can define the following equivalence relation: for each $r, s \in \mathcal{R}(G)$, we say that $r$ and $s$ are equivalent, denoted by $r\sim s$, if and only if for every finite set $F\subset V(G)$ we require that $r$ and $s$ have tails in the same connected component of $G\setminus F$. The relation $\sim$ is an equivalence relation, and equivalence classes are called ends. The space $\Omega (G) =\mathcal{R}(G)/\sim$ is called the end space of $G$. The definition of the ends of a graph as we know it today was presented by Halin in \cite{halin}.

In \cite{ciclofundamental}, Diestel and Kühn introduced the concept of an infinite cycle by extending the definition of a cycle. For that, let us define a topology on $\vert G\vert = G\cup \Omega (G)$. We begin by viewing $G$ itself (without ends) as
the point set of a $1$-complex. Then every edge is a copy of the real interval $[0,1]$, and we give it the corresponding metric and topology. For every vertex $v$ we choose as a basis of open neighborhoods the open stars of radius $\frac{1}{n}$.

Now let us define the basic open sets around the ends of $G$. For every finite subset $F\subset V(G)$ and every end $\varepsilon\in \Omega(G)$ there is exactly one component $C$ of $G\setminus F$ that contains a tail of every ray in $\varepsilon$. We denote this component by $C(F,\varepsilon)$. We say that $\varepsilon$ belongs to $C(F,\varepsilon)$. Define $\Omega (F,\varepsilon)=\lbrace \varepsilon'\in \Omega(G): C(F, \varepsilon') = C(F,\varepsilon) \rbrace$. Then, given $\varepsilon\in \Omega(G)$ and finite $F\subset V(G)$ we define the open neighborhood around $\varepsilon$ to be
$$\Omega(F,\varepsilon)\bigcup C(F,\varepsilon)\bigcup E(F,\varepsilon)$$
where $E(F,\varepsilon)$ is the set of all inner points of edges between $F$ and $C(F, \varepsilon)$. For more details on this topology, see \cite{ciclofundamental} and \cite{restrições1}. A set $C\subset \vert G\vert$ is a \textbf{circle} if it is homeomorphic to the unit circle, $\mathbb{S}^1$. Then $C$ contains all edges that contain an inner point that belongs to $C$. Thus, the graph consisting of these edges and their endvertices is the \textbf{cycle} defined by $C$. Every finite cycle in $G$ is also a cycle in this sense. 
With this definition, some double rays are examples of infinite cycles, namely precisely those two half-rays are equivalent with respect to $\sim$. 

In this paper, we extend the study of Salia's conjecture to infinite graphs by obtaining some partial results. The first natural case to study is that of locally finite infinite graphs, but we show in Proposition \ref{prop0} that it is not possible to have such a graph which is dHp. Subsequently, we investigate Salia's conjecture restricted to dHp graphs where the locally finite property is relaxed to only one side of the bipartition. Initially, we consider dHp graphs where the side $A$ is locally finite and impose restrictions on the degrees of the vertices of $B$. These restrictions are inspired by some results over finite graphs in \cite{barát2023double}, where it is assumed that, for every vertex $v$ in $B$, the degree of $v$ in $G$ is either 2 or equal to the $\vert A\vert $. Adapting these ideas to infinite graphs, we obtain the following results:

\begin{thm}
    Let $G(A,B)$ be a bipartite infinite dHp graph with infinite sides $A$ and $B$, where $A$ is locally finite. If for each $v\in B$, either $N(v)$ is finite or $N(v)=A\setminus F_v$ for some finite $F_v\subset A$, then there is a family $\mathcal{C}_X$ of disjoint cycles in $G$ such that $V(\bigcup\mathcal{C}_X)\cap A=X$.   
\end{thm}

\begin{thm}
Let $G(A,B)$ be a bipartite infinite dHp graph with infinite sides $A$ and $B$, where $A$ is locally finite. If $B$ has only finitely many vertices with finite degree, then there is a collection of pairwise disjoint 2-regular subgraphs in $G$ covering all the vertices of $A$.    
\end{thm}

These results are obtained through combinatorial constructions and by using the concept of $F$-limit of cycles (for more details, see \cite{F-limit}). Later, in Section \ref{sec2}, we study the case when the set $B$ is locally finite. With this simple restriction, we obtain the following result:

\begin{lemma}
    Let $G(A,B)$ be a bipartite infinite dHp graph with infinite countable sides $A$ and $B$, where $B$ is locally finite. Then there exists an infinite cycle in $G$ covering all vertices of $A$.
\end{lemma}

Note that, in the previous result, we impose a restriction on the cardinality of the set $A$. When $A$ is uncountable, we can use elementary submodels to decompose the graph $G$ into countable dHp subgraphs. This decomposition allows us to obtain the following result:

\begin{thm}
     Let $G(A,B)$ be a bipartite infinite dHp graph with infinite sides $A$ and $B$, where $B$ is locally finite. Then there exists a collection of pairwise disjoint 2-regular subgraphs in $G$ that covers all vertices of $A$.
\end{thm}

In Section \ref{sec3}, we focus on studying Salia's conjecture in countable graphs, without the restriction that one of the sides is locally finite. At the beginning of the section, we introduce a concept that will serve as a hypothesis for one of our positive results: let $G$ be an infinite graph. We say that a vertex $v \in V(G)$ is \textbf{pseudo-isolated} if there exists a finite subset $F \subset V(G) \setminus \{v\}$ such that $\{w \in V(G\setminus F): d_{G\setminus F}(v,w) = 2\} =\emptyset$.

\begin{thm}
    Let $G(A,B)$ be a bipartite infinite dHp graph with infinite countable sides $A$ and $B$. If the set of vertices of $A$ that are pseudo-isolated in $G$ is finite, then there is a collection of pairwise disjoint 2-regular subgraphs in $G$ that covers all vertices of $A$.
\end{thm}

From this result, we can derive Corollary \ref{cor1}, which provides another partial positive answer to Salia's conjecture by imposing restrictions on the degrees of the vertices of $B$, similar to the partial results in Section \ref{sec2}. Afterwards, we present a result showing that Salia's conjecture is false for countable graphs in general:

\begin{thm}
    The double Hall property on infinite bipartite graphs $G(A,B)$ does not imply the existence of a perfect $A$-matching.     
\end{thm}

We conclude Section \ref{sec3} with another partial positive result by imposing a restriction on the degrees of the vertices of $B$.

\begin{thm}
        Let $G(A,B)$ be a bipartite infinite dHp graph with infinite sides $A$ and $B$, and for each $v\in B$ either $deg(v)=2$ or $N(v)=A$. If Salia's conjecture is true for finite graphs, then there exists a cycle in $G$ covering all vertices of $A$.
\end{thm}

\section{Locally finite dHp graphs}\label{sec2}

In what follows we assume that a dHp graph is connected, since the subgraph induced by $A\cup N_{G}(A)$ is always connected. 

\begin{prop}\label{prop0}
    If a bipartite infinite graph $G(A,B)$ is dHp, then it is non-locally finite. 
\end{prop}

\begin{proof}
    Suppose $G(A,B)$ is a locally finite dHp infinite graph. If $A$ is finite, then $B$ must be infinite. Thus, since $N(w)\neq \emptyset$ for all $w\in B$ there is $v\in A$ such that $d(v)$ is infinite, which is a contradiction. If $A$ is infinite, let $v\in A$. For each $v'\in A$ with $v\neq v'$, there is at least one $w\in B$ such that $w\in N(v)\cap N(v')$. Since each $w\in B$ has finite degree then there is only finitely many $v'\in A$ with $v'\neq v$ such that $w\in N(v)\cap N(v')$. Thus, $v$ has infinite degree, and this is a contradiction.  
\end{proof}

However, it is possible to assume that one of the sides of the bipartite graph $G(A,B)$ is locally finite. 

\begin{prop}\label{prop2.2}
    Let $G(A,B)$ be a bipartite infinite graph dHp with infinite sides $A$ and $B$. If $A$ is locally finite, then $A$ is countable. Consequently, if in addiction $G(A,B)$ is connected, then $B$ is also countable.
\end{prop}

\begin{proof}
    Suppose by contradiction that $A$ is uncountable. Consider $\mathcal{F}=\lbrace N(v): v\in A \rbrace $ the family of sets of neighbors of each vertex of $A$. By the $\Delta$-system Lemma there is an uncountable $\mathcal{F}'\subset \mathcal{F}$ that forms a $\Delta$-system of root $R$. That is, for all distinct $N(v), N(u)\in \mathcal{F}'$, $N(v)\cap N(u) = R$, for some finite $R\subset B$.

    Let $\vert R \vert = n$. Let $X\in [A]^{n+1}$ be such that if $v\in X$ then $N(v)\in \mathcal{F}'$. Note that $N^2(X)\subset R$. Therefore, $\vert N^2(X)\vert\leq \vert R\vert =n < n+1 = \vert X\vert $, which is a contradiction.
\end{proof}

Since throughout the paper we will always assume that the graphs $G(A,B)$ are connected, by Proposition \ref{prop2.2}, if $A$ is locally finite and $G(A,B)$ is a dHp graph, then $G(A,B)$ is countable.

Unlike the finite case, in which the ideal solution to Salia's conjecture would be given by a cycle, in the case of infinite dHp graphs, it is generally not possible to obtain a solution that is connected. In particular, there are cases in which the only possible solutions must contain infinitely many connected components; as the following example shows.

\begin{ex}\label{ex2.3}
\emph{First consider the bipartite graph $H$ whose sides are $A=\lbrace u_i : i\in\mathbb{N}\rbrace$ and $B=\lbrace v_i : i\in\mathbb{N}\rbrace$. Let us define the edges of the $H$ as follows: for all $i\in\mathbb{N}$ we have $u_iv_j\in E(H)$ for all $0\leq j\leq i+1$, see Figure \ref{fig1}.}

\begin{figure}[ht]
    \centering
\begin{tikzpicture}
  \draw[white] (-0.8,-0.2) circle (0.01) node[above right,red] {\footnotesize $B$};
  \draw[fill=red] (0,0) circle (0.08) node[above right] {\footnotesize };
  \draw[fill=red] (1,0) circle (0.08) node[above right] {\footnotesize };
  \draw[fill=red] (2,0) circle (0.08) node[above right] {\footnotesize };
  \draw[fill=red] (3,0) circle (0.08) node[above right] {\footnotesize };
  \draw[fill=red] (4,0) circle (0.08) node[above right] {\footnotesize };
  \draw[fill=black] (4.3,0) circle (0.01) node[above right] {\footnotesize };
  \draw[fill=black] (4.5,0) circle (0.01) node[above right] {\footnotesize };
  \draw[fill=black] (4.7,0) circle (0.01) node[above right] {\footnotesize };
  
  \draw[white] (-0.8,1.8) circle (0.01) node[above right,blue] {\footnotesize $A$};
  \draw[fill=blue] (0,2) circle (0.08) node[above right] {\footnotesize };
  \draw[fill=blue] (1,2) circle (0.08) node[above right] {\footnotesize };
  \draw[fill=blue] (2,2) circle (0.08) node[above right] {\footnotesize };
  \draw[fill=blue] (3,2) circle (0.08) node[above right] {\footnotesize };
  \draw[fill=blue] (4,2) circle (0.08) node[above right] {\footnotesize };
  \draw[fill=black] (4.3,2) circle (0.01) node[above right] {\footnotesize };
  \draw[fill=black] (4.5,2) circle (0.01) node[above right] {\footnotesize };
  \draw[fill=black] (4.7,2) circle (0.01) node[above right] {\footnotesize };
  
  \draw[-] (0,2) -- (0,0);
  \draw[-] (1,2) -- (0,0);
  \draw[-] (2,2) -- (0,0);
  \draw[-] (3,2) -- (0,0);
  \draw[-] (4,2) -- (0,0);
  \draw[-] (0,2) -- (1,0);
  \draw[-] (1,2) -- (1,0);
  \draw[-] (2,2) -- (1,0);
  \draw[-] (3,2) -- (1,0);
  \draw[-] (4,2) -- (1,0);
  
  \draw[fill=black] (4.3,1) circle (0.01) node[above right] {\footnotesize };
  \draw[fill=black] (4.5,1) circle (0.01) node[above right] {\footnotesize };
  \draw[fill=black] (4.7,1) circle (0.01) node[above right] {\footnotesize };
  
  \draw[-] (1,2) -- (2,0);
  \draw[-] (2,2) -- (2,0);
  \draw[-] (3,2) -- (2,0);
  \draw[-] (4,2) -- (2,0);
  
  \draw[-] (2,2) -- (3,0);
  \draw[-] (3,2) -- (3,0);
  \draw[-] (4,2) -- (3,0);

  \draw[-] (3,2) -- (4,0);
  \draw[-] (4,2) -- (4,0);
  
\end{tikzpicture}
    \caption{$\Gamma$ (graph $H$)}
    \label{fig1}
\end{figure}
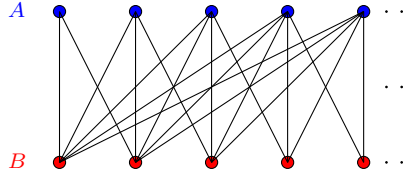

\emph{Let $\lbrace H_i : i\in\mathbb{N}\rbrace$ be a collection of pairwise disjoint copies of the graph $H$ whose sides are denoted by $A_i$ and $B_i$. }

\emph{Let $S=\bigcup_{j\in\mathbb{N}}A_j$.  Let us consider $Y=\lbrace y_i : i\in\mathbb{N}\rbrace$ and define the graph $\Gamma$ as follows: Let $V(\Gamma)=\bigcup_{n\in\mathbb{N}}V(H_n)\sqcup Y$ and let $E(\Gamma)=\bigcup_{n\in\mathbb{N}}E(H_n)\cup\bigcup_{i\in\mathbb{N}}\lbrace y_iv : v\in A_n \text{ with } i-1\leq n\rbrace$, see Figure \ref{fig2}. Note that, $\Gamma$ is a bipartite graph with sides $S=\bigcup_{n\in\mathbb{N}}A_n$ and $R=\bigcup_{n\in\mathbb{N}}B_n\cup Y$.}

\begin{figure}[ht]
    \centering
\begin{tikzpicture}[scale=0.65]

\draw[white] (-1,-0.2) circle (0.01) node[above right,red] {\footnotesize $B_0$};
  \draw[fill=red] (0,0) circle (0.08) node[above right] {\footnotesize };
  \draw[fill=red] (1,0) circle (0.08) node[above right] {\footnotesize };
  \draw[fill=red] (2,0) circle (0.08) node[above right] {\footnotesize };
  \draw[fill=red] (3,0) circle (0.08) node[above right] {\footnotesize };
 
  \draw[fill=black] (3.3,0) circle (0.01) node[above right] {\footnotesize };
  \draw[fill=black] (3.5,0) circle (0.01) node[above right] {\footnotesize };
  \draw[fill=black] (3.7,0) circle (0.01) node[above right] {\footnotesize };
  
    \draw[white] (-1,1) circle (0.01) node[above right,blue] {\footnotesize $H_0$};
  \draw[white] (-1,1.8) circle (0.01) node[above right,blue] {\footnotesize $A_0$};
  \draw[fill=blue] (0,2) circle (0.08) node[above right] {\footnotesize };
  \draw[fill=blue] (1,2) circle (0.08) node[above right] {\footnotesize };
  \draw[fill=blue] (2,2) circle (0.08) node[above right] {\footnotesize };
  \draw[fill=blue] (3,2) circle (0.08) node[above right] {\footnotesize };
  
  \draw[fill=black] (3.3,2) circle (0.01) node[above right] {\footnotesize };
  \draw[fill=black] (3.5,2) circle (0.01) node[above right] {\footnotesize };
  \draw[fill=black] (3.7,2) circle (0.01) node[above right] {\footnotesize };
  
  \draw[-] (0,2) -- (0,0);
  \draw[-] (1,2) -- (0,0);
  \draw[-] (2,2) -- (0,0);
  \draw[-] (3,2) -- (0,0);
  
  \draw[-] (0,2) -- (1,0);
  \draw[-] (1,2) -- (1,0);
  \draw[-] (2,2) -- (1,0);
  \draw[-] (3,2) -- (1,0);

  \draw[fill=black] (3.3,1) circle (0.01) node[above right] {\footnotesize };
  \draw[fill=black] (3.5,1) circle (0.01) node[above right] {\footnotesize };
  \draw[fill=black] (3.7,1) circle (0.01) node[above right] {\footnotesize };
  
  \draw[-] (1,2) -- (2,0);
  \draw[-] (2,2) -- (2,0);
  \draw[-] (3,2) -- (2,0);

  \draw[-] (2,2) -- (3,0);
  \draw[-] (3,2) -- (3,0);

  \draw[white] (4,-0.2) circle (0.01) node[above right,red] {\footnotesize $B_1$};
  \draw[fill=red] (5,0) circle (0.08) node[above right] {\footnotesize };
  \draw[fill=red] (6,0) circle (0.08) node[above right] {\footnotesize };
  \draw[fill=red] (7,0) circle (0.08) node[above right] {\footnotesize };
  \draw[fill=red] (8,0) circle (0.08) node[above right] {\footnotesize };
 
  \draw[fill=black] (8.3,0) circle (0.01) node[above right] {\footnotesize };
  \draw[fill=black] (8.5,0) circle (0.01) node[above right] {\footnotesize };
  \draw[fill=black] (8.7,0) circle (0.01) node[above right] {\footnotesize };
  
    \draw[white] (4,1) circle (0.01) node[above right,blue] {\footnotesize $H_1$};
  \draw[white] (4,1.8) circle (0.01) node[above right,blue] {\footnotesize $A_1$};
  \draw[fill=blue] (5,2) circle (0.08) node[above right] {\footnotesize };
  \draw[fill=blue] (6,2) circle (0.08) node[above right] {\footnotesize };
  \draw[fill=blue] (7,2) circle (0.08) node[above right] {\footnotesize };
  \draw[fill=blue] (8,2) circle (0.08) node[above right] {\footnotesize };
  
  \draw[fill=black] (8.3,2) circle (0.01) node[above right] {\footnotesize };
  \draw[fill=black] (8.5,2) circle (0.01) node[above right] {\footnotesize };
  \draw[fill=black] (8.7,2) circle (0.01) node[above right] {\footnotesize };
  
  \draw[-] (5,2) -- (5,0);
  \draw[-] (6,2) -- (5,0);
  \draw[-] (7,2) -- (5,0);
  \draw[-] (8,2) -- (5,0);
  
  \draw[-] (5,2) -- (6,0);
  \draw[-] (6,2) -- (6,0);
  \draw[-] (7,2) -- (6,0);
  \draw[-] (8,2) -- (6,0);

  \draw[fill=black] (8.3,1) circle (0.01) node[above right] {\footnotesize };
  \draw[fill=black] (8.5,1) circle (0.01) node[above right] {\footnotesize };
  \draw[fill=black] (8.7,1) circle (0.01) node[above right] {\footnotesize };
  
  \draw[-] (6,2) -- (7,0);
  \draw[-] (7,2) -- (7,0);
  \draw[-] (8,2) -- (7,0);

  \draw[-] (7,2) -- (8,0);
  \draw[-] (8,2) -- (8,0);

    \draw[white] (9,-0.2) circle (0.01) node[above right,red] {\footnotesize $B_2$};
  \draw[fill=red] (10,0) circle (0.08) node[above right] {\footnotesize };
  \draw[fill=red] (11,0) circle (0.08) node[above right] {\footnotesize };
  \draw[fill=red] (12,0) circle (0.08) node[above right] {\footnotesize };
  \draw[fill=red] (13,0) circle (0.08) node[above right] {\footnotesize };
 
  \draw[fill=black] (13.3,0) circle (0.01) node[above right] {\footnotesize };
  \draw[fill=black] (13.5,0) circle (0.01) node[above right] {\footnotesize };
  \draw[fill=black] (13.7,0) circle (0.01) node[above right] {\footnotesize };
  
    \draw[white] (9,1) circle (0.01) node[above right,blue] {\footnotesize $H_2$};
  \draw[white] (9,1.8) circle (0.01) node[above right,blue] {\footnotesize $A_2$};
  \draw[fill=blue] (10,2) circle (0.08) node[above right] {\footnotesize };
  \draw[fill=blue] (11,2) circle (0.08) node[above right] {\footnotesize };
  \draw[fill=blue] (12,2) circle (0.08) node[above right] {\footnotesize };
  \draw[fill=blue] (13,2) circle (0.08) node[above right] {\footnotesize };
  
  \draw[fill=black] (13.3,2) circle (0.01) node[above right] {\footnotesize };
  \draw[fill=black] (13.5,2) circle (0.01) node[above right] {\footnotesize };
  \draw[fill=black] (13.7,2) circle (0.01) node[above right] {\footnotesize };
  
  \draw[-] (10,2) -- (10,0);
  \draw[-] (11,2) -- (10,0);
  \draw[-] (12,2) -- (10,0);
  \draw[-] (13,2) -- (10,0);
  
  \draw[-] (10,2) -- (11,0);
  \draw[-] (11,2) -- (11,0);
  \draw[-] (12,2) -- (11,0);
  \draw[-] (13,2) -- (11,0);

  \draw[fill=black] (13.3,1) circle (0.01) node[above right] {\footnotesize };
  \draw[fill=black] (13.5,1) circle (0.01) node[above right] {\footnotesize };
  \draw[fill=black] (13.7,1) circle (0.01) node[above right] {\footnotesize };
  
  \draw[-] (11,2) -- (12,0);
  \draw[-] (12,2) -- (12,0);
  \draw[-] (13,2) -- (12,0);

  \draw[-] (12,2) -- (13,0);
  \draw[-] (13,2) -- (13,0);

      \draw[white] (14,-0.2) circle (0.01) node[above right,red] {\footnotesize $B_3$};
  \draw[fill=red] (15,0) circle (0.08) node[above right] {\footnotesize };
  \draw[fill=red] (16,0) circle (0.08) node[above right] {\footnotesize };
  \draw[fill=red] (17,0) circle (0.08) node[above right] {\footnotesize };
  \draw[fill=red] (18,0) circle (0.08) node[above right] {\footnotesize };
 
  \draw[fill=black] (18.3,0) circle (0.01) node[above right] {\footnotesize };
  \draw[fill=black] (18.5,0) circle (0.01) node[above right] {\footnotesize };
  \draw[fill=black] (18.7,0) circle (0.01) node[above right] {\footnotesize };
  
    \draw[white] (14,1) circle (0.01) node[above right,blue] {\footnotesize $H_3$};
  \draw[white] (14,1.8) circle (0.01) node[above right,blue] {\footnotesize $A_3$};
  \draw[fill=blue] (15,2) circle (0.08) node[above right] {\footnotesize };
  \draw[fill=blue] (16,2) circle (0.08) node[above right] {\footnotesize };
  \draw[fill=blue] (17,2) circle (0.08) node[above right] {\footnotesize };
  \draw[fill=blue] (18,2) circle (0.08) node[above right] {\footnotesize };
  
  \draw[fill=black] (18.3,2) circle (0.01) node[above right] {\footnotesize };
  \draw[fill=black] (18.5,2) circle (0.01) node[above right] {\footnotesize };
  \draw[fill=black] (18.7,2) circle (0.01) node[above right] {\footnotesize };
  
  \draw[-] (15,2) -- (15,0);
  \draw[-] (16,2) -- (15,0);
  \draw[-] (17,2) -- (15,0);
  \draw[-] (18,2) -- (15,0);
  
  \draw[-] (15,2) -- (16,0);
  \draw[-] (16,2) -- (16,0);
  \draw[-] (17,2) -- (16,0);
  \draw[-] (18,2) -- (16,0);

  \draw[fill=black] (18.3,1) circle (0.01) node[above right] {\footnotesize };
  \draw[fill=black] (18.5,1) circle (0.01) node[above right] {\footnotesize };
  \draw[fill=black] (18.7,1) circle (0.01) node[above right] {\footnotesize };
  
  \draw[-] (16,2) -- (17,0);
  \draw[-] (17,2) -- (17,0);
  \draw[-] (18,2) -- (17,0);

  \draw[-] (17,2) -- (18,0);
  \draw[-] (18,2) -- (18,0);

  \draw[fill=red] (1.5,4) circle (0.08) node[above] {\footnotesize ${y_0}$};
  \draw[fill=red] (6.5,4) circle (0.08) node[above] {\footnotesize ${y_1}$};
  \draw[fill=red] (11.5,4) circle (0.08) node[above] {\footnotesize ${y_2}$};
  \draw[fill=red] (16.5,4) circle (0.08) node[above] {\footnotesize ${y_3}$};
   
  \draw[-] (1.5,4) -- (0,2);
  \draw[-] (1.5,4) -- (1,2);
    \draw[-] (1.5,4) -- (2,2);
  \draw[-] (1.5,4) -- (3,2);
   \draw[-] (1.5,4) -- (5,2);
  \draw[-] (1.5,4) -- (6,2);
    \draw[-] (1.5,4) -- (7,2);
  \draw[-] (1.5,4) -- (8,2);
  \draw[-] (1.5,4) -- (10,2);
  \draw[-] (1.5,4) -- (11,2);
    \draw[-] (1.5,4) -- (12,2);
  \draw[-] (1.5,4) -- (13,2);
   \draw[-] (1.5,4) -- (15,2);
  \draw[-] (1.5,4) -- (16,2);
    \draw[-] (1.5,4) -- (17,2);
  \draw[-] (1.5,4) -- (18,2);

  \draw[-] (6.5,4) -- (0,2);
  \draw[-] (6.5,4) -- (1,2);
    \draw[-] (6.5,4) -- (2,2);
  \draw[-] (6.5,4) -- (3,2);
   \draw[-] (6.5,4) -- (5,2);
  \draw[-] (6.5,4) -- (6,2);
    \draw[-] (6.5,4) -- (7,2);
  \draw[-] (6.5,4) -- (8,2);
  \draw[-] (6.5,4) -- (10,2);
  \draw[-] (6.5,4) -- (11,2);
    \draw[-] (6.5,4) -- (12,2);
  \draw[-] (6.5,4) -- (13,2);
   \draw[-] (6.5,4) -- (15,2);
  \draw[-] (6.5,4) -- (16,2);
    \draw[-] (6.5,4) -- (17,2);
  \draw[-] (6.5,4) -- (18,2);

   \draw[-] (11.5,4) -- (5,2);
  \draw[-] (11.5,4) -- (6,2);
    \draw[-] (11.5,4) -- (7,2);
  \draw[-] (11.5,4) -- (8,2);
  \draw[-] (11.5,4) -- (10,2);
  \draw[-] (11.5,4) -- (11,2);
    \draw[-] (11.5,4) -- (12,2);
  \draw[-] (11.5,4) -- (13,2);
   \draw[-] (11.5,4) -- (15,2);
  \draw[-] (11.5,4) -- (16,2);
    \draw[-] (11.5,4) -- (17,2);
  \draw[-] (11.5,4) -- (18,2);

  \draw[-] (16.5,4) -- (10,2);
  \draw[-] (16.5,4) -- (11,2);
    \draw[-] (16.5,4) -- (12,2);
  \draw[-] (16.5,4) -- (13,2);
   \draw[-] (16.5,4) -- (15,2);
  \draw[-] (16.5,4) -- (16,2);
    \draw[-] (16.5,4) -- (17,2);
  \draw[-] (16.5,4) -- (18,2);

  \draw[fill=black] (19.3,1) circle (0.02) node[above right] {\footnotesize };
  \draw[fill=black] (19.5,1) circle (0.02) node[above right] {\footnotesize };
  \draw[fill=black] (19.7,1) circle (0.02) node[above right] {\footnotesize };

  \draw[fill=black] (19.3,3) circle (0.02) node[above right] {\footnotesize };
  \draw[fill=black] (19.5,3) circle (0.02) node[above right] {\footnotesize };
  \draw[fill=black] (19.7,3) circle (0.02) node[above right] {\footnotesize };

    \draw[fill=black] (19.3,4) circle (0.02) node[above right] {\footnotesize };
  \draw[fill=black] (19.5,4) circle (0.02) node[above right] {\footnotesize };
  \draw[fill=black] (19.7,4) circle (0.02) node[above right] {\footnotesize };

\end{tikzpicture}
    \caption{The graph $\Gamma$}
    \label{fig2}
\end{figure}
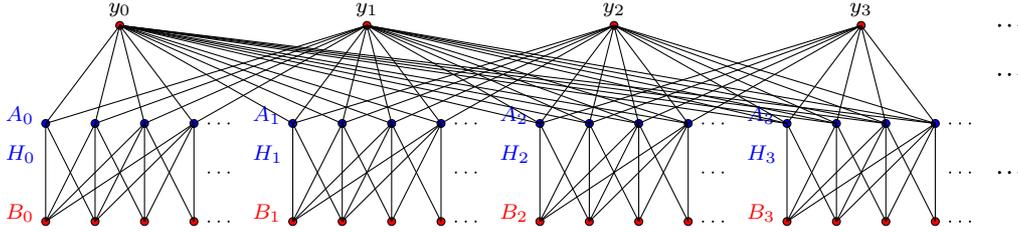
\end{ex}

\begin{prop}\label{prop2.5}
    The graph $\Gamma=\Gamma(S,R)$ constructed in Example \ref{ex2.3} is an infinite countable bipartite dHp graph with side $S$ locally finite. Furthermore, every collection of pairwise disjoint 2-regular subgraphs contained in $\Gamma$ that covers $S$ has infinitely many connected components.
\end{prop}

\begin{proof}
   Fix $X\in [S]^{\aleph_0}$ and let $J=\lbrace j\in\mathbb{N}: X\cap A_j\neq \emptyset \rbrace$ and $J_1=\lbrace j\in\mathbb{N}: \vert X\cap A_j\vert = 1 \rbrace$. Three cases are possible:
   \begin{enumerate}
       \item $J_1=\emptyset$.

       Then $N^2_{\Gamma}(X)\supset \bigcup_{j\in J} N^2_{\Gamma}(X\cap A_j)$ and $\vert N^2_{\Gamma}(X\cap A_j)\vert \geq \vert N^2_{H_j}(X\cap A_j)\vert\geq \vert X\cap A_j\vert$ because $H_j$ is dHp and $\vert X\cap A_j\vert \geq 2$.

       \item $\vert J_1\vert$, $J_1=\lbrace j\rbrace$.

       Then, as in Case 1, $\vert N^2_{\Gamma}(X\setminus A_j)\cap (\bigcup_{n\in\mathbb{N}}B_n)\vert\geq \vert X\setminus A_j \vert = \vert X\vert - 1,$
       but also $y_0\in N^2_{\Gamma}(X)$, so $N^2_{\Gamma}(X)\geq (\vert X\vert - 1)+1 = \vert X\vert$.

       \item $J_1=\lbrace j_0, \dots , j_k \rbrace$, where $j_0<j_1< \cdots < j_k$.

       Then in the same way as in Case 1 we can prove that $$\vert N^2_{\Gamma}(X\setminus \bigcup_{s\leq k} A_{j_s})\cap (\bigcup_{n\in\mathbb{N}}B_n)\vert\geq \vert X\setminus \bigcup_{s\leq k} A_{j_s}\vert = \vert X\vert - (k+1).$$ 
       On the other hand, it is easy to check that $$\lbrace y_0, \dots, y_{j_0}, y_{j_0+1}\rbrace \cup \lbrace y_{j_s}, y_{j_s+1} : 1\leq s\leq k-1 \rbrace \subset N^2_{\Gamma}(X\cap \bigcup_{0\leq s\leq k} A_{j_s}),$$ and the number of these $y's$ is at least $k+1$. Thus, $\vert N^2_{\Gamma}(X)\vert \geq \vert X\vert - (k+1)+(k+1)= \vert X\vert$.
   \end{enumerate}
   
    Furthermore, each $H_i$ can be separated from the others by finitely many vertices. Thus, given a double ray $r$ in $S$ and $i\in\mathbb{N}$, either $r$ has a tail entirely contained in $H_i$ or $r$ intersects $H_i$ in only finitely many vertices. Therefore, $S$ cannot be covered by a finite number of finite cycles or double ray. Thus, every collection of pairwise disjoint 2-regular subgraphs contained in $\Gamma$ that covers $S$ must have infinitely many components. 
    
\end{proof}

Due to Proposition \ref{prop2.5} the search for answers to Salia's conjecture begins by looking for coverings given by collection of pairwise disjoint 2-regular subgraphs. The following theorem is one of the main results in \cite{barát2023double} (Theorem 2.6). 

\begin{thm}[Barát, Grzesik, Jung, Nagy, and Pálvölgyi - \cite{barát2023double}]\label{thm2.6}
    Let $G$ be a finite bipartite dHp graph. If $deg(v)\in \lbrace 2, \vert A \vert \rbrace$ for each $v\in B$, then there exists a cycle in $G$ covering all the vertices of $A$.
\end{thm}

Let $G(A,B)$ be a bipartite graph with sides $A$ and $B$. To extend Theorem \ref{thm2.6} to infinite graphs, we will add hypotheses to $A$ and $B$ when they are infinite. We begin showing that $A$ cannot be locally finite under these circumstances. 

\begin{prop}\label{prop2.6}
    Let $G(A,B)$ be a bipartite infinite dHp graph. If, for each $v\in B$, $N(v)$ is finite or $N(v)=A$, then $A$ is not locally finite.
\end{prop}

\begin{proof}
    Consider $B^{\infty}=\lbrace v\in B : N(v)=A \rbrace $. If $B^{\infty}$ is infinite, then all vertices of $A$ are neighbors of infinitely many vertices of $B$, so $A$ is non-locally finite. If $B^{\infty}$ is finite, let $n=\vert B^{\infty}\vert$. Suppose that $A$ is locally finite. Thus, there exists $X\subset A$ with $\vert X\vert = n+1 $ such that $(N(v)\cap N(w))\setminus B^{\infty}=\emptyset$ for all $v,w \in X$. Now, note that $N^2(X)\subset B^{\infty}$, therefore $\vert N^2(X)\vert =\vert B^{\infty}\vert = n < n+1 =\vert X\vert$, which is a contradiction.
\end{proof}

However, if we weaken the restriction on the of degree of the vertices of $B$, from ``being a neighbor of the entire set $A$'' to ``being a neighbor of  almost all $A$'' then it becomes possible for the existence of bipartite infinite dHp graphs where $A$ is locally finite, these are treated in Theorem \ref{thm1}. Before presenting this result, we will need to see the concept of $F$-limit of cycles that was introduced in \cite{F-limit}.

Intuitively, an $F$-limit $H$ on a graph $G$ is a subgraph such that there exists a sequence of subgraphs $\langle H_n\rangle_{n\in\mathbb{N}}$ of $G$ where a vertex or edge of $G$ is in $H$ if there are many $n\in\mathbb{N}$ to which the vertex or edge belongs to $H_n$. As usual, we will use an ultrafilter to be precise about what we mean by many.

\begin{defn}
    We say that $F\subset \wp(\mathbb{N})$ is a \textbf{filter} over $\mathbb{N}$ if:
     \begin{enumerate}[(a)]
        \item $\emptyset \notin F$;
        \item If $a,b\in F$ then $a\cap b \in F$;
        \item If $a\in F$ and $b\supset a$ then $b\in F$;
      \end{enumerate}
    We say that $F$ is an \textbf{ultrafilter} if $F$ is maximal (that is, if $G\supset F$ is a filter then $G=F$). We also say that an ultrafilter $F$ is \textbf{non-principal} if it does not contain finite subsets of $\mathbb{N}$ (see e.g. \cite{Jechbook}). 
\end{defn}

\begin{defn}
    Let $F$ be a non-principal ultrafilter over $\mathbb{N}$, let $G$ be a graph and let $\langle H_n\rangle_{n\in\mathbb{N}}$ be a sequence of subgraphs of $G$. 
    We say that a subgraph $H\subset G$ is the \textbf{$F$-limit} of $\langle H_n\rangle_{n\in\mathbb{N}}$ if $H=(V_H,E_H)$, where
    $$V_H=\lbrace v\in V(G) : \lbrace n : v\in V(H_n) \rbrace \in F \rbrace ,$$
    $$E_H=\lbrace e\in E(G) : \lbrace n : e\in E(H_n) \rbrace \in F \rbrace.$$
    
    In this case, we say that the sequence $\langle H_n\rangle_{n\in\mathbb{N}}$ is \textbf{convergent} and \textbf{converges} to $H$. In this work we will use $F$-limit of $\langle H_n\rangle_{n\in\mathbb{N}} $ where $H_n$ is a finite family of finite cycles for each $n\in\mathbb{N}$ .
    
\end{defn}    

To prove Lemma \ref{lemma} we will create a sequence of finite subgraphs whose limit is the subgraph we want. In order for each connected component of the desired subgraph in Lemma \ref{lemma} to have the required properties we will use the following result of \cite{barát2023double} (Theorem 3.2):

\begin{thm}[Barát, Grzesik, Jung, Nagy, and Pálvölgyi - \cite{barát2023double}]\label{thm3.2}
    Let $G$ be a bipartite dHp graph with sides $A$ and $B$. For every $X\subset A$, $\vert X\vert \geq 2$, there is a family $\mathcal{C}_X$ of disjoint cycles in $G$ such that $V(\bigcup\mathcal{C}_X)\cap A=X$.
\end{thm}

Consider $B'=\lbrace v\in B : N(v)=A\setminus F_v \text{ for some } F_v\subset [A]^{<\aleph_0}\rbrace $, where $[A]^{<\aleph_0}$ is the set of all finite subsets of $A$.

\begin{lemma}\label{lemma}
     Let $G(A,B)$ be a bipartite infinite dHp graph with infinite sides $A$ and $B$, where $A$ is locally finite. The following statements are true:

     \vspace{0.2cm}
     
     \hspace{-0.5cm}(a) There is a subgraph $H$ of $G$ which covers all vertices of $A$, each vertex of $V(H)\cap A$ has degree equal to $2$ and each vertex of $V(H)\cap B$ has degree less than or equal to $2$. Furthermore, the vertices of $V(H)\cap B$ that have degree equal to $1$ in $H$ have infinite degree in $G$.

     \vspace{0.2cm}

     \hspace{-0.5cm}(b) If for each $v\in B$ either $N(v)$ is finite or there exists a finite set $ F_v\subset A$ such that $N(v)=A\setminus F_v$, then $B'$ must be infinite. 

     \vspace{0.2cm}
     
     \hspace{-0.5cm}(c) Under the assumptions of (b) there exists a subgraph $H$ of $G$ with properties stated in (a), which in addition contains infinitely many vertices of $B'$.
     
     \vspace{0.2cm}
     
     \hspace{-0.5cm}(d)  Let $H$ be the subgraph from item (a) and let $S\subset \lbrace v\in B\cap V(H): deg_H(v) =1 \rbrace$. If there is an infinite  $\overline{A}\subset A$ such that for each $u\in \overline{A}$ there exists $w_u\in B$ such that $uw_u\in E(H)$ and $N_G(w_u)\cap \overline{A}$ is infinite, and for each $v\in S$ we have $N_G(v)\cap\overline{A}$ is infinite, then there is a collection of pairwise disjoint 2-regular subgraphs in $G$ covering all the vertices of $A$.
\end{lemma}

\begin{proof}
    \textbf{(a):} Consider $A=\lbrace u_i : i\in\mathbb{N}\rbrace$, $A_n=\lbrace u_0, \cdots , u_n\rbrace$, $B_n= N^2(A_n)$ and $G_n=[ A_n\cup B_n]$ the subgraph induced by $A_n\cup B_n$ for each $n\geq 1$. Note that $G_n$ is a finite bipartite dHp graph with sides $A_n$ and $B_n$ for each $n\in\mathbb{N}$.

    By Theorem \ref{thm3.2}, each $G_n$ contains a collection of pairwise disjoint cycles that covers all vertices of $A_n$. Let us denote it by $C_n$. Given $v\in A$, there exists $n\in\mathbb{N}$ such that $v\in A_n$. Therefore $v\in V(C_m)$ for all $m\geq n$. Since $v$ has finite degree, there exists a connected component $H_v$ of the $F$-limit of $\langle C_n:n\in\mathbb{N}\rangle$ that contains $v$. Let $H$ be the union of all such $H_v$ for $v$ in $A$.

    Let $v\in A$ and suppose that $N(v)=\lbrace w_1, \cdots , w_k \rbrace $. For each $w_i, w_j\in N(v)$ with $i<j$ consider $U_{ij}= \lbrace n\in\mathbb{N}: vw_i, vw_j \in E(C_n)\rbrace$. Since $\bigcup_{ij=1}^k U_{ij}= \omega\in F$ and $F$ is an ultrafilter, then there are $i<j$ such that $U_{ij}\in F$. Therefore, $vw_i, vw_j\in E(H)$, so every vertex of $A$ has degree exactly equal to $2$ in $H$.

    Suppose there is $v\in B\cap V(H)$ with degree greater than $2$ in $H$. Let $u,w,t\in A$ be such that $vu, vw, vt\in E(H)$. Then, by the definition of $F$-limit, there is $n\in\mathbb{N}$ such that $vu, vw, vt \in E(C_n)$, and this is a contradiction, since $C_n$ is the union of disjoint pairwise cycles. Therefore, each vertex of $B\cap V(H)$ have a degree less than or equal to $2$ in $H$. Let $v\in B\cap V(H)$ with degree equal to 1 in $H$. Suppose by contradiction that $v$ has finite degree in $G$. Let $N(v)=\lbrace w_1, \cdots , w_k \rbrace $. For each $w_i, w_j\in N(v)$ with $i<j$ consider $U_{ij}= \lbrace n\in\mathbb{N}: vw_i, vw_j \in E(C_n)\rbrace$. Since $\bigcup_{ij=1}^k U_{ij}=\omega\in F$ and $F$ is an ultrafilter, then there are $i<j$ such that $U_{ij}\in F$. Therefore, $vw_i, vw_j\in E(H)$, which is a contradiction.
\vspace{0.2cm}

    \textbf{(b):} Suppose that $B'$ is finite. For every $v\in A$ we set $A_v=\lbrace w\in A: N^2(\lbrace v,w\rbrace)\setminus B'\neq \emptyset \rbrace$. Since $v\in A$ has finite degree, then $A_v$ is finite for all $v\in A$. Hence, there exists $X\subset A$ such that, for any $v,w\in X$, $v\notin A_w$ and $\vert X\vert > \vert B'\vert$. Hence, $N^2(X)=B'$ and, therefore, $\vert N^2(X)\vert <\vert X\vert$ which is a contradiction.

\vspace{0.2cm}

    \textbf{(c):} Let $H$ be such as in (a). If $H$ contains at most finitely many vertices of $B'$, then $W=\lbrace v\in A: vu\in E(H) \text{ for some } u\notin B'\rbrace$ is an infinite set. Let $W=\lbrace u_i : i\in\mathbb{N}\rbrace$ be an enumeration such that for every $i\in\mathbb{N}$ there is $w_i\notin B'$ such that $u_{2i} w_i, u_{2i+1}w_i\in E(H)$. Consider $B'\setminus V(H)=\lbrace v_i : i\in\mathbb{N}\rbrace$. Since $v_0\in B'$, then there exists $i_0\in \mathbb{N}$ such that $u_{2i_0}, u_{2i_0+1}\in N(v_0)$. Then, remove the edges $u_{2i_0}w_{i_0}, u_{2i_0+1}w_{i_0}$ from $H$ and add the edges $u_{2i_0}v_0, u_{2i_0+1}v_0$ in $H$. Suppose we inductively add all $v_0, \cdots , v_k$ into $H$. Now let us add $v_{k+1}$ into $H$. Since $v_{k+1}\in B'$ there exists $i_{k+1}\in \mathbb{N}$ with $i_{k+1}> i_k$ such that $u_{2i_{k +1}}, u_{2i_{k+1}+1}\in N(v_{k+1})$. Then remove the edges $u_{2i_{k+1}}w_{i_{k+1}}, u_{2i_{k+1}+1}w_{i_{k+1}}$ from $H$ and add the edges $u_{2i_{k+1}}v_{k+1}, u_{2i_{k+1}+1}v_{k+1}$ into $H$. Therefore, by induction it is possible to obtain that $V(H)\cap B'$ is infinite.

       \vspace{0.2cm}
       
        \textbf{(d):} Consider the enumeration $S=\lbrace v_{2n}: n\in\mathbb{N}\rbrace$ if $S$ is infinite, and $S=\lbrace v_{2n}: n\leq k\rbrace$ for some $k\in\mathbb{N}$ if $S$ is finite. For each $v_{2n}\in S$ denote by $x_{2n}\in A$ the only vertex such that $x_{2n}v_{2n}\in E(H)$. As $v_0\in S$, $N_G(v_0)\cap \overline{A}$ is infinite. So, there is $u_0\in \overline{A}$ such that $v_0u_0\in E(G)$ and $v_0u_0\notin E(H)$. Then, set $H_0'=(V(H), (E(H)\setminus \lbrace u_0w_{u_0}\rbrace)\cup \lbrace v_0u_0\rbrace)$, and $A_0=\lbrace x_0, u_0\rbrace$. If $w_{u_0}$ has degree 0 in $H_0'$, then set $H_0=(V(H_0')\setminus \lbrace w_{u_0}\rbrace, E(H_0'))$. If $w_{u_0}$ has degree 1 in $H_0'$, then set $H_0=H_0'$. Let $v_1=w_{u_0}$ and denote by $x_{1}\in A$ the only vertex such that $x_{1}v_{1}\in E(H_0)$. 

        Since $N(v_1)\cap \overline{A}\setminus A_0$ is infinite, there exists $u_1\in \overline{A}\setminus A_0$ such that $v_1u_1\in E(G)$ and $v_1u_1\notin E(H_0)$. Then, set $H_1'=(V(H_0), (E(H_0)\setminus \lbrace u_1w_{u_1}\rbrace)\cup \lbrace v_1u_1\rbrace)$. Set $A_1=A_0 \cup \lbrace x_1, u_1\rbrace$. If $w_{u_1}$ has degree 0 in $H_1'$, then set $H_1=(V(H_1')\setminus \lbrace w_{u_1}\rbrace, E(H_1'))$. If $w_{u_1}$ has degree 1 in $H_1'$, then set $H_1=H_1'$. Let $v_3=w_{u_1}$ be and denote by $x_3\in A$ the only vertex from which $v_3x_3\in E(H_1)$.

        Suppose that for all $m\in \mathbb{N}$ with $0\leq m < k$ we have constructed $H_m$ and $A_m$ as in the previous steps. If there is no $v_k$ then choice $H_k=H_{k-1}$ and $A_k=A_{k-1}$. If there is $v_k$, then either $v_k\in S$ or $v_k=w_{u_i}$ for some $0\leq i \leq k-1$. Hence, $N(v_k)\cap (\overline{A}\setminus A_{k-1})$ is infinite. So, there is $u_k\in \overline{A}\setminus A_{k-1}$ such that $v_ku_k\in E(G)$ and $v_ku_k\notin E (H_{k-1})$. Then, set $H_k'=(V(H_{k-1}), (E(H_{k-1})\setminus \lbrace u_kw_{u_k}\rbrace)\cup \lbrace v_ku_k\rbrace)$, and $A_k=A_{k-1} \cup \lbrace x_k, u_k\rbrace$. If $w_{u_k}$ has degree 0 in $H_k'$, then set $H_k=(V(H_k')\setminus \lbrace w_{u_k}\rbrace, E(H_k'))$. If $w_{u_k}$ has degree 1 in $H_k'$, then set $H_k=H_k'$. Let $v_{n_k}=w_{u_k}$, where $n_k\in \mathbb{N}$ is the first odd number to which a vertex has not yet been assigned, and denote by $x_{n_k}\in A$ the only vertex such that $v_{n_k}x_{n_k}\in E(H_k)$.

        Inductively, we have for all $k\in\mathbb{N}$ a subgraph $H_k$ of $G$ in which all vertices $\lbrace v_0, \cdots , v_k\rbrace$ have degree equal to 2. Consider $C $ the $F$-limit of $\langle H_n: n\in\mathbb{N}\rangle$. Note that all vertices of $C$ have degree equal to 2 in $C$. Indeed, if $v\in V(C)\cap \lbrace v_n : n\in\mathbb{N}\rbrace$, then $v=v_m$ for some $m\in\mathbb{N}$ and therefore, for every $k\geq m$, $vu_m, vx_m\in E(H_k)$. Therefore, $vu_m, vx_m\in E(C)$. Furthermore, these are the only edges of $v$ in all $H_k$ for $k\geq m$. If $v\in V(C)\cap A_m$ for some $m\in\mathbb{N}$, then the neighbors of $v$ are the same two vertices for all $H_k$ with $k\geq m$ . Therefore, $v$ has degree equal to 2 in $C$. If $v\in V(C)\setminus (\lbrace v_m: m\in\mathbb{N}\rbrace \cup \bigcup_{m\in\mathbb{N}} A_m)$, then there are only $s, t \in V(G)$ such that $vs, vt \in E(H)$ and $vs, vt\in E(H_m)$ for all $m\in\mathbb{N}$. Therefore, $vs, vt\in E(C)$, therefore, $v$ has degree 2 in $C$. Thus, the collection of connected components of $C$ forms the desired cover.

\end{proof}

In the proof of Lemma \ref{lemma} (a), it would be intuitive to apply Hall's Theorem twice to construct the subgraph $H$, rather than using the $F$-limit. However, this approach if it is not done carefully might yield a subgraph where some vertices in $B \cap V(H)$ of degree 1 in $H$ have finite degree in $G$.

\begin{thm}\label{thm1}
    Let $G(A,B)$ be a bipartite infinite dHp graph with infinite sides $A$ and $B$, where $A$ is locally finite. If for each $v\in B$, either $N(v)$ is finite or $N(v)=A\setminus F_v$ for some finite $F_v\subset A$, then there is a collection of pairwise disjoint 2-regular subgraphs in $G$ covering all the vertices of $A$.    
\end{thm}

\begin{proof} 
        Since $A$ is locally finite, there exists a subgraph $H$ of $G$ with the properties guaranteed by Lemma \ref{lemma} (a).
        
        Furthermore, each vertex of $B\cap V(H)$ of degree $1$ in $H$ belongs to the set $B'$. Thus, every connected component of $H$ is either a double ray or a simple ray with an initial vertex in $B'$ or a finite cycle or a finite path with final vertices in $B'$. By  Lemma \ref{lemma} (b), $B'$ is infinite. By Lemma \ref{lemma} (c), without loss of generality we may assume that $B'\cap V(H)$ is infinite.

        Consider $\overline{A}=\lbrace u\in A : \exists v\in B'\text{ such that } vu\in E(H)\rbrace $ and $$S=\lbrace v \in B'\cap V(H) : deg_H(v)=1\rbrace \subset \lbrace v\in V(H): deg_H(v)=1\rbrace.$$ 
        For each $u\in \overline{A}$, denote by $w_u$ a vertex in $B'\cap V(H)$ such that $uw_u\in E(H)$. Note that, for each $u\in \overline{A}$ the set $N(w_u)\cap \overline{A}$ is infinite.  Furthermore, for each $v\in S$, the set $N(v)\cap \overline{A}$ is infinite. Therefore, we can apply Lemma \ref{lemma} (d) to obtain the desired cover.

\end{proof}

\begin{thm}\label{thm1.1}
    Let $G(A,B)$ be a bipartite infinite dHp graph with infinite sides $A$ and $B$, where $A$ is locally finite. If $B$ has only finitely many vertices with finite degree, then there is a collection of pairwise disjoint 2-regular subgraphs in $G$ covering all the vertices of $A$.
\end{thm}

\begin{proof}
        Let $\overline{B}$ be the subset of vertices of $B$ which have infinite degree. 
        Since $A$ is locally finite, we can apply Lemma \ref{lemma} (a) and obtain a subgraph $H$ of $G$ with the properties guaranteed by Lemma \ref{lemma} (a).
        
        Furtheremore, each vertex of $B\cap V(H)$ of degree $1$ in $H$ belongs to the set $\overline{B}$. Thus, every connected component of $H$ is either a double ray or a simple ray with an initial vertex in $\overline{B}$ or a finite cycle or a finite path with final vertices in $\overline{B}$.    

        Consider $T=\lbrace u\in A : \exists v\in B\setminus \overline{B} \text{ with } vu\in E(H)\rbrace $, $\overline{A}= A\setminus T$ and $S=\lbrace v \in B\cap V(H) : deg_H(v)=1\rbrace $. Note that $\overline{A}$ is cofinite in $A$. For each $u\in \overline{A}$, there is $w_u\in \overline{B}$ such that $uw_u\in E(H)$. Hence, for each $u\in \overline{A}$ the set $N(w_u)\cap A$ is infinite. Furthermore, for each $v\in S$ the set $N(v)\cap \overline{A}$ is infinite. Therefore, we can apply the Lemma \ref{lemma} item (d) to obtain the desired cover.

\end{proof}

In the paper \cite{lavrov2023halltypeconditionpathcovers},  Lavrov and Vandenbussche searched for Hall-type conditions for the existence of a path cover for bipartite graphs. With this, a weaker version of Salia's conjecture is presented. Theorem \ref{theorem2.15} presented below provides a sufficient condition for a certain class of infinite graphs, which implies the validity of the conjecture, for finite graphs, presented in \cite{lavrov2023halltypeconditionpathcovers}.

\begin{thm}\label{theorem2.15}
    If every bipartite infinite dHp graph $G(A,B)$ with infinite sides $A,B$, where $A$ is locally finite and for each $v\in B$ either $N(v)$ is finite or there is $F_v\subset A$ finite such that $N(v)=A\setminus F_v$, admits a double ray $A$-cover, then every bipartite finite dHp graph $M(X,Y)$ admits a path in $M$ covering all the vertices of $X$.
\end{thm}

\begin{proof}
    Let $M(X,Y)$ be a finite bipartite dHp graph with sides $X$ and $Y$. Let $G(A,B)$ also be an infinite bipartite dHp graph with sides $A$ and $B$ in which for each $v\in B$, either $N(v)$ is finite or $N(v)= A\setminus F_v$ for some finite subset $F_v\subset A$. Consider a new graph $$G_1=(V(G)\cup V(M)\cup \lbrace v_0, v_1\rbrace, E(G)\cup E(M)\cup \lbrace v_iu : i=0,1 \text{ and } u\in A\cup X\rbrace)$$
    with $v_0,v_1\notin V(G)\cup V(M)$. Note that, $G_1$ is an infinite bipartite graph with sides $A_1=A\cup X$ and $B_1=B\cup Y \cup \lbrace v_0, v_1\rbrace$. Furthermore, if $v\in B_1$ then either $N_{G_1}(v)$ is finite or $N_{G_1}(v)=A_1\setminus F_v$ for some finite subset $F_v\subset A_1$. Finally, note that $G_1$ is a dHp graph. Indeed, let $S\subset A_1$ with $\vert S\vert \geq 2$. If $S\subset A$, then as $G$ is a dHp graph, $\vert N^2(S)\vert \geq \vert S\vert$ in $G$ and therefore in $G_1$. If $S\subset X$, then as $M$ is a dHp graph, $\vert N^2(S)\vert \geq \vert S\vert$ in $M$ and therefore in $G_1$. If $\vert S\cap A \vert \neq 1$ and $\vert S\cap X\vert \neq 1$, then $\vert N^2_{G_1}(S)\vert \geq \vert N^2_{G}(S\cap A)\vert + \vert N_M^2(S\cap X)\vert \geq \vert S\cap A \vert + \vert S\cap X\vert = \vert S \vert$. If $\vert S\cap A\vert =\vert S\cap X\vert =1$, then $N^2_{G_1}(S)=\lbrace v_0, v_1\rbrace$, therefore $\vert N^ 2_{G_1}(S)\vert \geq \vert S\vert$. If $\vert S\cap A\vert =1$ and $\vert S\cap X\vert >1$, then since $M$ is a dHp graph $\vert N^2_M(S\cap X)\vert \geq \vert S\cap X\vert $. Hence, $\vert N^2_{G_1}(S)\vert = \vert \lbrace v_0, v_1\rbrace \vert + \vert N^2_M(S\cap X)\vert \geq \vert\lbrace S\cap A\rbrace\vert + \vert S\cap X\vert= \vert S\vert$. If $\vert S\cap A\vert >1$ and $\vert S\cap X\vert =1$, then an argument analogous to the previous case shows that $\vert N^2_{G_1}(S)\vert \geq \vert S\vert$. Therefore, $G_1$ is a dHp graph.

    So, by hypothesis, $G_1$ admits a double ray $A_1$-cover. We denote it by $D$. Note that there is only one element of $D$ which contains the vertices of $X$. This occurs due to the fact that the only vertices that connect the graph $M$ to the graph $G$ are the vertices $v_0$ and $v_1$ which only let a single double ray enter and leave $M\subset G_1$ . Let $R\in D$ be such that $R\supset X$. Then, $R\cap M$ is a finite path that covers all vertices of $X$ as desired.
\end{proof}

Next, we shall consider the case when $B$ is locally finite.

\begin{lemma}\label{lemma22}
    Let $G(A,B)$ be a bipartite infinite dHp graph with infinite countable sides $A$ and $B$, where $B$ is locally finite. Then there exists an infinite cycle in $G$ covering all vertices of $A$.
\end{lemma}

\begin{proof}
    Since $A$ is countable let $A=\lbrace v_i : i\in\mathbb{N}\rbrace $. Since $G$ is a dHp graph, then there are distinct $u_1, u_2\in B$ such that $u_1\in N(v_0)\cap N(v_1)$ and $u_2\in N(v_0)\cap N (v_2)$. Thus, the paths $P_1=v_0u_1v_1$ and $P_2=v_2u_2v_0u_1v_1$ are contained in $G$. Furthermore, $P_2$ extends $P_1$ and let us denote $P_2=P_1\cup  v_2u_2v_0$.

    Suppose inductively that for all $m \in \mathbb{N}$ with $1 < m \leq k$, $P_m$ is a path with endpoints $v_{n_m}$ and $v_{n_{m-1}}$ that extends $P_{m-1}$. Here, $v_{n_m}$ is the vertex in $A \setminus P_{m-1}$ with the smallest index. Let $v_{n_{k+1}}\in A\setminus P_k$ be the smallest index. Let us construct a path $P_{k+1}$ that extends $P_k$ and with final vertices $v_{n_k}$ and $v_{n_{k+1}}$.

    Since $B$ is locally finite, for all but finitely many $v_1\in A$ we have $$N(v_{n_{k-1}})\cap N(v_s)\cap (V(P_k)\cap B)=N(v_{n_{k+1}})\cap N(v_s)\cap (V(P_k)\cap B)=\emptyset.$$ 
    Pick such a $v_s$, $u_s\in N(v_{n_{k-1}})\cap N(v_s)$, $w_s\in N(v_{n_{k+1}})\cap N(v_s)\setminus \lbrace u_s\rbrace$, and note that $P_{k+1}=P_k\cup v_{n_{k-1}}u_sv_sw_sv_{n_{k+1}}$ is as required.
    
\end{proof}

\begin{thm} \label{thm2}
    Let $G(A,B)$ be a bipartite infinite dHp graph with infinite sides $A$ and $B$, where $B$ is locally finite. Then there exists a collection of pairwise disjoint 2-regular subgraphs in $G$ that covers all vertices of $A$.
\end{thm}

\begin{proof}
    Let $\kappa =\vert B\vert $ and $B=\lbrace u_i : i\in\kappa \rbrace$ be a bijective enumeration. Let $M_0$ be a countable elementary submodel containing all relevant objects (see e.g. \cite{Soukup}) and such that every $u_i \in M_0$ for $i < \omega$. By elementarity, $G_0 = G \cap M_0$ is a dHp graph, with sides $A \cap M_0$ and $B \cap M_0$, where $B \cap M_0$ is locally finite. Since $G_0$ is countable, by Lemma \ref{lemma22}, there is an infinite cycle $C_0$ that covers $A \cap M_0$. 

Proceeding by induction, suppose that each $M_\xi$ is defined for $\xi < \eta$. Note that, if $w \in B \cap M_\xi$ for some $\xi$, $N(w) \subset M_\xi$. Therefore, $G \setminus \bigcup_{\xi < \eta} M_\xi$ is a dHp graph. Let $M_\eta$ be a countable elementary submodel that contains the $\omega$ first $u_i$ of $G \setminus \bigcup_{\xi < \eta} M_\xi$. Define $G_\eta = M_\eta \cap (G \setminus \bigcup_{\xi < \eta} M_\xi)$. Again by elementarity, $G_\eta$ is a dHp graph with the needed hypothesis and we can apply Lemma \ref{lemma22} to obtain an infinite cycle $C_{\eta}$ that covers $A \cap M_\eta$.

Note that $\bigcup_{\xi < \kappa} C_\xi$ is a collection of disjoint infinite cycles, which is the desired.

\end{proof}

\section{Non-locally finite countable dHp graphs}\label{sec3}

In this section we will work with infinite dHp countable graphs without the restriction that one of the sides is locally finite. 

\begin{defn}
Let $G$ be an infinite graph. We say that a vertex $v\in V(G)$ is \textbf{pseudo-isolated} if there exists a finite $F\subset V(G)\setminus \lbrace v\rbrace$ such that $\{w \in V(G\setminus F): d_{G\setminus F}(v,w) = 2\} =\emptyset$.    
\end{defn}

\begin{defn}
Let $G(A, B)$ be a bipartite infinite dHp graph with infinite countable sides $A$ and $B$. Let $<$ be an order on $A$ with 
order type $\omega$. We call a double ray $R=...v_3u_3v_1u_1v_0$ $u_0v_2u_2v_4u_4...$ with $v_0\in A$
\emph{$<$-\textbf{economical} (with starting point $v_0$)}, if
\begin{itemize}
\item for every $k\in\mathbb{N}$, $v_{2k+2}$ is the $<$-minimal vertex in $A$
which can be added to 
$$R\upharpoonright[-(2k+1),2k]:=v_{2k+1}u_{2k+1}\ldots v_3u_3v_1u_1v_0u_0v_2u_2v_4u_4\ldots u_{2k-2}v_{2k}$$
on the $v_{2k}$-side, and
\item for every $k\in\mathbb{N}$, $v_{2k+3}$ is the $<$-minimal vertex in $A$
which can be added to 
$$R\upharpoonright[-(2k+1),2k+2]:=v_{2k+1}u_{2k+1}\ldots v_1u_1v_0u_0v_2u_2\ldots u_{2k-2}v_{2k}u_{2k}v_{2k+2}$$
on the $v_{2k+1}$-side.
\end{itemize}    
\end{defn}

\begin{lemma}\label{lemma3.1}
Let $G(A,B)$, $<$ and $R$ be as above. Suppose that
$G$ has no pseudo-isolated vertices in $A$ and
$R$ is $<$-economical. Then the  restriction 
$G'$ of $G$ to $A'=A\setminus\{v_i:i\in\mathbb{N}\}$ and
$B'=B\setminus\{u_i:i\in\mathbb{N}\}$ has no pseudo-isolated vertices.
\end{lemma}
\begin{proof}
Suppose, contrary to our claim, that $v\in A'$ is pseudo-isolated
in $G'(A',B')$. Since $v$ is not pseudo-isolated in $G$,
there exist injective sequences $b_0,b_1,\ldots$ in $B$ and
$a_0,a_1,\ldots$ in $A$ such that $vb_i$ and $b_ia_i$ are edges in $G(A,B)$. Since $v$ is  pseudo-isolated in $G'$,
there exists $i_0\in\mathbb{N}$ such that for all $i\geq i_0$, either
$b_i\in\{u_j:j\in\mathbb{N}\}$ or $a_i\in\{v_j:j\in\mathbb{N}\}$.
Without loss of generality we may assume that for each $i\geq i_0$,
if $b_i=u_j$  then $v<v_j,v_{j+2}$, and if $a_i=v_j$ then $v<v_{j+2}$.  
Thus, for any $i\geq i_0$, two cases are possible.

I. $b_i=u_j$ for some $j$ such that $v<v_j,v_{j+2}$.
But this contradicts $R$ being economical since in the construction of $R$, depending on the parity of $j$, we would have to add $v$ instead of $v_j$ or $v_{j+2}$, but $v\not\in R$.

II. $b_i\not\in\{u_j:j\in\mathbb{N}\}$ and $a_i=v_j$ for some $v<v_{j+2}$.
Again, this contradicts $R$ being economical since in the construction of $R$ we would have to add $v$ instead of  $v_{j+2}$, but $v\not\in R$.

This contradiction completes our proof.
\end{proof}

\begin{thm}\label{thm3.1}
    Let $G(A,B)$ be a bipartite infinite dHp graph with infinite countable sides $A$ and $B$. If the set of vertices of $A$ that are pseudo-isolated in $G$ is finite, then there is a collection of pairwise disjoint 2-regular subgraphs in $G$ that covers all vertices of $A$.
\end{thm}

\begin{proof}
Let $A'$ be the finite set of vertices of $A$ that are pseudo-isolated. Then there exists a finite $D\subset N^2(A')$ such that the induced subgraph $H=G[A'\cup D]$ is dHp. As $H$ is finite then by Theorem \ref{thm3.2} there is a collection of pairwise disjoint 2-regular subgraphs $L$ in $H$ that covers $A'$. Since $L$ is finite being a subset of $H$ for $G'=G\setminus L$ we have that no $a\in A\cap V(G')$ is pseudo-isolated in $G'$. Let $<$ be an order on $A\cap V(G')$ with order type $\omega$. Let $a_0$ be the $<$-minimal element and $R_0$ be a minimal double ray in $G'$ with starting point $a_0$. Set $G_0=G'\setminus R_0$. By Lemma \ref{lemma3.1} we have that no $a\in A\cap V(G_0)$ is pseudo-isolated.

Suppose that we have constructed a sequence $\langle R_i : i\leq n \rangle$ of mutually disjoint double rays in $G'$ such that $G_n\coloneqq G'\setminus \bigcup_{i\leq n} R_i$ has no elements of $A\cap V(G_i)$ which are pseudo-isolated in $G_n$. Then let $a_{n+1}$ be the $<$-minimal element of $A\cap V(G_n)$, $R_{n+1}$ be a $<$-economical double ray of $G_n$ with starting point $a_{n+1}$, and $G_{n+1}\coloneqq G_n\setminus R_{n+1}$. This completes our inductive construction. Since each an was chosen in a minimal way, we have $A\cap G'\subset \bigcup_{n\in\mathbb{N}}R_n$.

    


\end{proof}

\begin{cor}\label{cor1}
     Let $G(A,B)$ be a bipartite infinite dHp graph with infinite sides $A$ and $B$. If for each $v\in B$ either $N(v)$ is finite or there is a finite subset $F_v\subset A$ such that $N(v)=A\setminus F_v$ and $B'=\lbrace v\in B: N(v)=A\setminus F_v \text{ for some } F_v\in [A]^{\aleph_0}\rbrace$ is finite, then there is a collection of pairwise disjoint 2-regular subgraphs in $G$ covering all the vertices of $A$.
\end{cor}

\begin{proof}
If $B'$ is finite, then we will prove that the set of pseudo-isolated vertices of $A$ in $G$ is finite. Suppose not. Denote by $A'$ the set of pseudo-isolated vertices of $A$ in $G$. If $v\in A'$ then there exists a finite set $S_v\subset V(G)\setminus \lbrace v\rbrace $ such that $\{w \in V(G\setminus S_v): d_{G\setminus S_v}(v,w) = 2\} =\emptyset$. Let $S_v^*=S_v\cap (B\setminus B')$, $S_A=S_v\cap A$. Since $S_v^*$ is finite, then $N_G(S_v^*)$ is finite. Thus, $N_G(S_v^*)\cup S_{A,v}\subset A$ is finite. Let $A_v=A\setminus (N_G(S_v^*)\cup S_{A,v} )$. 

Consider $v_0\in A'$. Since $A'$ is an infinite set, then $A'\cap A_{v_0}\neq \emptyset$. For all $w\in A'\cap A_{v_0}$, $N_G^2(\lbrace v_0, w\rbrace)\subset B'$. Let $v_1\in A'\cap A_{v_0}$. For each $n\in\mathbb{N}$, we have $A'\cap A_{v_0}\cap \cdots \cap A_{v_n}\neq \emptyset$, because $A'$ is an infinite set and each $A_{v_i}$ is cofinite. For all $w\in A'\cap A_{v_0}\cap \cdots \cap A_{v_n}$, $N_G^2(\lbrace v_0, w\rbrace)\subset B'$ and for all $v_i$ with $0\leq i< n$, $N_G^2(\lbrace v_i, v_n\rbrace)\subset B'$. By induction, there exists an infinite set $X=\lbrace v_i : i\in\mathbb{N}\rbrace\subset A'$ such that $N^2_G(X)\subset B'$, and this is a contradiction.
\end{proof}

Salia's conjecture does not hold for infinite countable bipartite graphs dHp $G(A,B)$ with infinite sides $A$ and $B$, where $A$ is non-locally finite and only finitely many vertices of $B$ have finite degree. Let us look at the counterexample that disproves the conjecture for this case.

\begin{ex}\label{thm3}
    There exists an infinite countable bipartite graph $G(A,B)$  with infinite sides $A,B$, where $A$ is non-locally finite only at two vertices, no vertices of $B$ have finite degree, which nonetheless has no perfect $A$-matching, and hence also does not satisfy the Salia conjecture. 
\end{ex}

\begin{proof}
Let $G(A,B)$ be the following infinite bipartite graph: $A=\lbrace v_i : i\in\mathbb{N}\rbrace$, $B=\lbrace u_i : i\in\mathbb{N}\rbrace$ and $V(G)=A\cup B$. Also define the set of edges of the graph $V$ as follows: $E(V)= \lbrace v_0u_i : i\in \mathbb{N}\rbrace\cup \lbrace v_1u_i : i\in \mathbb{N}\rbrace \cup \lbrace v_ju_i : j\in\mathbb{N} \text{ with } j>1 \text{ and } 0\leq i \leq j-1\rbrace$. Note that $A$ is non-locally finite only at two vertices $v_0, v_1$ and all vertices of $B$ have infinite degree.

\begin{figure}[ht]
    \centering
\begin{tikzpicture}
  \draw[white] (-0.8,-0.2) circle (0.01) node[above right,red] {\footnotesize $B$};
  \draw[fill=blue] (0,-2) circle (0.08) node[ right] {\footnotesize $v_0$};
  \draw[fill=blue] (1,-2) circle (0.08) node[ right] {\footnotesize $v_1$};
  \draw[fill=red] (0,0) circle (0.08) node[ right] {\footnotesize $u_0$};
  \draw[fill=red] (1,0) circle (0.08) node[ right] {\footnotesize $u_1$};
  \draw[fill=red] (2,0) circle (0.08) node[ right] {\footnotesize $u_2$};
  \draw[fill=red] (3,0) circle (0.08) node[ right] {\footnotesize $u_3$};
  \draw[fill=red] (4,0) circle (0.08) node[right] {\footnotesize $u_4$};
  \draw[fill=red] (5,0) circle (0.08) node[right] {\footnotesize $u_5$};
  \draw[fill=black] (5.5,0) circle (0.01) node[above right] {\footnotesize };
  \draw[fill=black] (5.7,0) circle (0.01) node[above right] {\footnotesize };
  \draw[fill=black] (5.9,0) circle (0.01) node[above right] {\footnotesize };
    \draw[fill=black] (3.5,-1) circle (0.01) node[above right] {\footnotesize };
  \draw[fill=black] (3.7,-1) circle (0.01) node[above right] {\footnotesize };
  \draw[fill=black] (3.9,-1) circle (0.01) node[above right] {\footnotesize };
  
  \draw[white] (-0.8,1.8) circle (0.01) node[above right,blue] {\footnotesize $A$};
  \draw[fill=blue] (0,2) circle (0.08) node[above right] {\footnotesize $v_2$};
  \draw[fill=blue] (1,2) circle (0.08) node[above right] {\footnotesize $v_3$};
  \draw[fill=blue] (2,2) circle (0.08) node[above right] {\footnotesize $v_4$};
  \draw[fill=blue] (3,2) circle (0.08) node[above right] {\footnotesize $v_5$};
  \draw[fill=blue] (4,2) circle (0.08) node[above right] {\footnotesize $v_6$};
  \draw[fill=black] (4.3,2) circle (0.01) node[above right] {\footnotesize };
  \draw[fill=black] (4.5,2) circle (0.01) node[above right] {\footnotesize };
  \draw[fill=black] (4.7,2) circle (0.01) node[above right] {\footnotesize };

  \draw[-] (0,0) -- (0,-2);
  \draw[-] (1,0) -- (0,-2);
  \draw[-] (2,0) -- (0,-2);
  \draw[-] (3,0) -- (0,-2);
  \draw[-] (4,0) -- (0,-2);
  \draw[-] (5,0) -- (0,-2);

  \draw[-] (0,0) -- (1,-2);
  \draw[-] (1,0) -- (1,-2);
  \draw[-] (2,0) -- (1,-2);
  \draw[-] (3,0) -- (1,-2);
  \draw[-] (4,0) -- (1,-2);
  \draw[-] (5,0) -- (1,-2);

  \draw[-] (0,2) -- (0,0);
  \draw[-] (1,2) -- (0,0);
  \draw[-] (2,2) -- (0,0);
  \draw[-] (3,2) -- (0,0);
  \draw[-] (4,2) -- (0,0);
  \draw[-] (0,2) -- (1,0);
  \draw[-] (1,2) -- (1,0);
  \draw[-] (2,2) -- (1,0);
  \draw[-] (3,2) -- (1,0);
  \draw[-] (4,2) -- (1,0);

  \draw[-] (1,2) -- (2,0);
  \draw[-] (2,2) -- (2,0);
  \draw[-] (3,2) -- (2,0);
  \draw[-] (4,2) -- (2,0);
  
  \draw[-] (2,2) -- (3,0);
  \draw[-] (3,2) -- (3,0);
  \draw[-] (4,2) -- (3,0);

  \draw[-] (3,2) -- (4,0);
  \draw[-] (4,2) -- (4,0);
  \draw[-] (4,2) -- (5,0);
  
\end{tikzpicture}
    \caption{Counterexample to Salia's conjecture for infinite graphs}
    \label{fig3}
\end{figure}
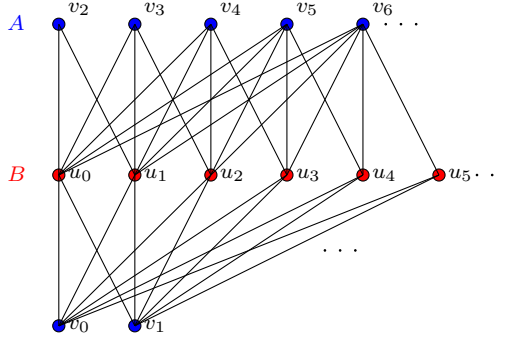

We claim that $G$ is a dHp graph. Instead, let $X\in [A]^{<\aleph_0}$, $X\setminus \lbrace v_0, v_1 \rbrace = \lbrace v_{k_1}, \dots , v_{k_n}\rbrace$ with $k_1\leq \cdots \leq k_n$, and note that $k_i\geq i+1$ for all $1\leq i \leq k$.   
 
$N^2(X)\supset \lbrace u_{k_1-2}, \cdots u_{k_{n-1} -2},$ $ u_{k_{n-1}-1}\rbrace $, and therefore $\vert N^2(X)\vert \geq \vert X\vert$. If $X$ contains exactly one of $v_0$ or $v_1$, then $N^2(X)=\lbrace u_0, u_1, \dots , u_{k_n-1}\rbrace$, and hence $\vert N^2(X)\vert =k_n\geq n+1 \geq \vert X\vert$. If $\lbrace v_0, v_1 \rbrace \subset X$, then $N^2(X)$ is infinite.

Suppose by contradiction that $G$ admits a perfect $A$-matching. Let us denote it by $M$. Then, there are $l_0, l_1\in \mathbb{N}$ such that $v_0u_{l_0}, v_1u_{l_1}\in M$. Set $k_1=max\lbrace l_0, l_1 \rbrace +1$ and note that, if $X=\lbrace v_i : 2\leq i\leq k_1+2\rbrace$, then $N(X)=\lbrace u_i : 0\leq i \leq k_1\rbrace$. Therefore, for all $v_i\in X$ we have that there must be $u_{j_i}\in N(X)$ such that $v_iu_{j_i}\in M$. Thus, $Y=\lbrace u_{j_i}: 2\leq i \leq k_1+2 \rbrace \cup \lbrace u_{l_0},u_{l_1}\rbrace \subset \lbrace u_i : 0\leq i \leq k_1\rbrace$. However, $\vert Y\vert = k_1+2 \geq k_1+1 =\vert \lbrace u_i : 0\leq i \leq k_1\rbrace \vert$, which is a contradiction. Therefore, $G$ does not admit a perfect $A$-matching and, therefore, $G$ does not admit a collection of cycles that covers all vertices of $A$.
\end{proof}

Before presenting the next result, let us introduce some definitions and notations on order trees.

\begin{defn}
A partially ordered set $(T,\leq)$ is called an \textbf{order tree} if it has a unique minimal element (called the \textbf{root}) and all subsets of the form $\lceil t \rceil = \lceil t \rceil_T := \lbrace t' \in T :t'\leq t\rbrace$
are well-ordered. Write  $\lfloor t \rfloor = \lfloor t \rfloor_T  := \lbrace t' \in T : t\leq t'\rbrace$. For subsets $X\subset T$ we abbreviate $\lfloor X\rfloor :=\bigcup_{t\in X}\lfloor t\rfloor$ and $\lceil X\rceil:=\bigcup_{t\in X}\lceil t\rceil$. A maximal chain in $T$ is called a \textbf{branch}. The set $T^i$ of all points at height $i$ is the $i$th \textbf{level} of $T$, and we write $$T^{<i} := \bigcup\lbrace T^j: j < i\rbrace \text{, } T^{\leq i} := \bigcup\lbrace T^j:j \leq i\rbrace, $$ $$ T^{>i} := \bigcup\lbrace T^j:j > i\rbrace \text{ and } T^{\geq i} := \bigcup\lbrace T^j:j \geq i\rbrace.$$ 
    
\end{defn}

\begin{thm}\label{thm4}
    Let $G(A,B)$ be a bipartite infinite countable dHp graph with infinite sides $A$ and $B$, and for each $v\in B$ either $deg(v)=2$ or $N(v)=A$. If Salia's conjecture is true for finite graphs, then there exists a cycle in $G$ covering all the vertices of $A$.
\end{thm}

\begin{proof}
    If $B^{\infty}=\lbrace v\in B: N(v)=A\rbrace$ is infinite then $G$ contains a cycle that trivially covers $A$. So let us consider the case when $B^{\infty}$ is finite. For each $v\in A$ let $A_v=\lbrace w\in A : N^2(\lbrace v,w\rbrace)\setminus B^{\infty} \neq \emptyset \rbrace$. Consider $\mathcal{A}=\lbrace v\in A : \vert A_v\vert <\aleph_0 \rbrace$. Since $B^{\infty}$ is finite, then by the double Hall property, $\mathcal{A}$ is finite. Let $T$ be a normal spanning tree of $G$. 
    
   \vspace{0.3cm}

    \textbf{Claim 1:} $A\cap T^{\leq m}$ is finite for every $m\in\mathbb{N}$.

    \begin{proof}
        For $t\in V(T)$ we shall denote by $S_k= \lbrace t'\in V(T): t<t' \text{ and } k=\vert \lbrace r\in V(T): t<r<t'\rbrace \vert\rbrace$, for every $k\in\mathbb{N}$. Thus, $S_0(t)$ is the family of all immediate successors of $t$ in $T$. By König's lemma it suffices to show that if $a\in A\cap V(T)$, then $S_1(a)$ is finite. For this we shall first show that if $b\in B$, then $S_0(b)$ is finite. Indeed, since elements of $S_0(b)$ are incomparable in $T$ and $T$ is normal, we get $N^2_G(S_0(b))\subset \lceil b\rceil$, and hence $\vert S_0(b)\vert \leq \vert \lceil b\rceil\vert$ by the dHp.

        Now suppose that, contrary to our claim, there exists $a\in A$ such that $S_1(a)$ is infinite. Since $S_1(a)=\bigcup_{b\in S_0(a)}S_0(b)$ and each $S_0(b)$ is finite, we conclude that $B_1\coloneqq \lbrace b\in S_0(a): S_0(b)\neq \emptyset\rbrace$ is infinite. For every $b\in B_1$, let us pick $a(b)\in S_0(b)$ and note that $N^2_G(\lbrace a(b): b\in B_1\rbrace)\subset \lceil a\rceil$ since any two different $a(b_0)$ and $a(b_1)$, $b_0\neq b_1$ in $B_1$, are incomparable in $T$. But this contradicts the dHp because $\lbrace a(b): b\in B_1\rbrace$ is infinite.
    \end{proof}

    As a consequence we get that $\widetilde{B}=\lbrace b\in B: S_0(b)\neq \emptyset\rbrace$ also has the property $\vert \widetilde{B}\cap T^{\leq m}\vert \leq \aleph_0$ for every $m\in\mathbb{N}$. Note that if $b\in B\cap T^{\leq m}$ and $b\notin \widetilde{B}$, then $S_0(b)\neq \emptyset$, and hence $b$ is not a neighbor in $G$ of any $a\in A\cap T^{>m}$. This will be crucial later in the proof.

\vspace{0.3cm}

    \textbf{Claim 2:} $\vert \Omega (G)\vert \leq \vert B^{\infty}\vert$.
\begin{proof}
Suppose by contradiction that $\vert \Omega(G)\vert > \vert B^{\infty}\vert = n$. Let $\varepsilon_0, \cdots, \varepsilon_n\in \Omega (G)$ with $\varepsilon_i \neq\varepsilon_j$ for all $i\neq j$. By definition of end, there exists a finite set $F\subset V(G)$ such that if $C(F,\varepsilon_i)$ is the connected component of $G\setminus F$ in which $\varepsilon_i$ lives then $ C(F,\varepsilon_i)\neq C(F,\varepsilon_j)$ for all $i\neq j$. That is, $F$ separates the ends $\varepsilon_0, \cdots \varepsilon_n$. Hence $B^{\infty}\subset F$ and $N((F\setminus B^{\infty})\cap B)$ is finite, since each vertex of $(F\setminus B^{\infty })\cap B$ has degree 2. Thus, as each component $C(F, \varepsilon)$ has infinitely many vertices of $A$, for each $i=0, \cdots , n$ there is a vertex $v_i\in V(C(F, \varepsilon))\cap A$ such that $v_i\notin N((F\setminus B^{\infty})\cap B)$. Therefore, $N^2(\lbrace v_i : 0\leq i \leq n\rbrace)= B^{\infty}$. On the other hand, $\vert N^2(\lbrace v_i : 0\leq i \leq n\rbrace)\vert =n+1 > n =\vert B^{\infty}\vert$, which is a contradiction.    
\end{proof}    

\textbf{Claim 3:} For each $v\in A\setminus \mathcal{A}$, $\lfloor v\rfloor\cap A $ is infinite.
\begin{proof}
    Since $v\notin \mathcal{A}$, $A_v$ is infinite, that is, there are infinitely many vertices $w\in A$ such that $N^2(\lbrace v, w\rbrace)\setminus B^ {\infty}\neq \emptyset$. For every $w\in A_v$ fix $b(w)\in B\setminus B^{\infty}$ connected to $v$ and $w$. Then $B(v)=\lbrace b(w): w\in A_v\rbrace$ is infinite, since each $b\notin B^{\infty}$ can serve as $b(w)$ for at most 2 $w\in A_v$. Since each $b(w)\in B(v)$ is comparable in $T$ with $v$, and $\lceil v\rceil$ is finite, $v<_T b(w)$ for all but finitely many $b(w)\in B(v)$. Let us note that if $v<_Tb(w)$, then $v$ is comparable in $T$ with $w$, since $w$ is comparable with $b(w)$. Since $\lceil v\rceil$ is finite, by dHp we have $v<_T w$ for all but finitely many $w\in A_v$.
\end{proof}

\textbf{Claim 4:} There exists $n_0\in \mathbb{N}$ such that every $v\in (T\setminus T^{\leq n_0})\cap A$ is on exactly one infinite branch of $T $.

\begin{proof}
Indeed, since $T$ is a normal spanning tree, then the infinite branches of the tree are in bijection with the ends of $G$. Since $B^{\infty}\cup \mathcal{A}\cup \bigcup_{v\in\mathcal{A}}A_v$ is finite, then there exists $n_0'\in\mathbb{N}$ such that $B^{\infty}\cup \mathcal{A}\cup \bigcup_{v\in\mathcal{A}}A_v\subset T^{< n_0'}$. Since $G$ has only finitely many ends, there exists $n_0\in\mathbb{N}$, $n_0'<n_0$, such that each infinite component of $T\setminus T^{\leq n_0}$ contains only a single infinite branch. Let $v\in (T\setminus T^{\leq n_0})\cap A$. So $v\notin \mathcal{A}$. By Claim 3, $\lfloor v\rfloor \cap A$ is infinite, so it must be in a infinite branch. Furthermore, as $T^{\leq n_0}$ separates the ends of $G$, $v$ belongs to a single infinite branch.
\end{proof}

Let $\lbrace v_j^{n_0}: j<m\rbrace$ be an enumeration of $A\cap T^{n_0}$. By Claim 4 we have $m=\vert \Omega (G)\vert$. For every $j<m$ let $L_j$ be the unique infinite branch trough $T$ containing $v_j^{n_0}$. Let $n_1>n_0$ be big enough so that 

\begin{equation}\label{eq1}
   \vert A\cap L_j \cap (T^{<n_1}\setminus T^{\leq n_0})\vert > 2 \vert A \cap T^{<n_0}\vert + 2 \vert \widetilde{B}\cap T^{<n_0}\vert = 2 \vert T^{<n_0}\vert 
\end{equation}
for every $j<m$. Such an $n_1$ exists because the unique element of $L_j\cap T^k$ belongs to $A$ if and only if $k-(n_0-1)$ is odd (remember that $T^{n_0-1}\subset B$). We can assume that $n_1-n_0$ is even, so $T^{n_1-1}\subset B$. It follows that $G(A,B)$ restricted to $T^{<n_1}$ satisfies the dHp. Indeed, if $b\in B\setminus T^{<n_1}$, then $b\in B\setminus T^{\leq n_1}$ and hence for the immediate predecessor $a$ of $b$ in $T$ we have $a\in A \setminus T^{<n_1}$. Thus, $b\notin N^2_G(A\cap T^{<n_1})$, which yields $N^2_G(A\cap T^{< n_1})\subset B\cap T^{<n_1}$ which proves dHp of $G$ restricted to $T^{<n_1}$.

By out assumption there exists a cycle $C$ in $G$ restricted to $T^{<n_1}$ covering $A\cap T^{<n_1}$. Since every element of $A\cap C$ has exactly two other elements of $A\cap C$ at distance 2 in $C$ and exactly two elements of $B\cap C$ as neighbors in $C$, there exists $w(j)\in A\cap L_j\cap (T^{<n_1}\setminus T^{\leq n_0})$ having no elements of $A\cap T^{< n_0}$ at distance 2 in $C$, and no elements of $B\cap T^{<n_0}$ as neighbors in $C$. This is a direct consequence of (\ref{eq1}). Let $u(j)\in B$ and $v(j)\in A$ be such that $w(j)u(j)v(j)\subset C$. Thus, $u(j), v(j)\notin T^{< n_0}$, and hence $v(j), w(j)\notin \mathcal{A}\cup \bigcup_{v\in\mathcal{A}}A_v$ and $u(j)\notin B^{\infty}$. Since $u(j)$ has exactly two neighbors in $G$, these must be $w(j)$ and $v(j)$, and one of them, say $v(j)$, is the immediate predecessor of $u(j)$ in $T$. By the normality of $T$ we know that $w(j)$ is also comparable in $T$ with $u(j)$, and hence also it is comparable with $v(j)$. Let us fix $j<m$ and as long as we work below within $\lfloor v_j^{n_0}\rfloor$, we shall simply write $w, u, v$ instead of $w(j), u(j) $ and $v(j)$. Two cases are possible:

\begin{figure}[ht]
    \centering

\tikzset{every picture/.style={line width=0.75pt}} 

\begin{tikzpicture}[x=0.65pt,y=0.65pt,yscale=-0.7,xscale=0.7]

\draw    (151.09,15.6) -- (151.04,29.49) ;
\draw [shift={(151.1,12.6)}, rotate = 90.2] [fill={rgb, 255:red, 0; green, 0; blue, 0 }  ][line width=0.08]  [draw opacity=0] (8.93,-4.29) -- (0,0) -- (8.93,4.29) -- cycle    ;

\draw    (151.04,29.49) -- (150.98,46.37) ;
\draw [shift={(151.04,29.49)}, rotate = 90.2] [color={rgb, 255:red, 0; green, 0; blue, 0 }  ][fill={rgb, 255:red, 0; green, 0; blue, 0 }  ][line width=0.75]      (0, 0) circle [x radius= 3.35, y radius= 3.35]   ;

\draw    (150.98,46.37) -- (150.92,63.26) ;
\draw [shift={(150.98,46.37)}, rotate = 90.2] [color={rgb, 255:red, 0; green, 0; blue, 0 }  ][fill={rgb, 255:red, 0; green, 0; blue, 0 }  ][line width=0.75]      (0, 0) circle [x radius= 3.35, y radius= 3.35]   ;

\draw    (150.92,63.26) -- (150.86,80.15) ;
\draw [shift={(150.92,63.26)}, rotate = 90.2] [color={rgb, 255:red, 0; green, 0; blue, 0 }  ][fill={rgb, 255:red, 0; green, 0; blue, 0 }  ][line width=0.75]      (0, 0) circle [x radius= 3.35, y radius= 3.35]   ;
\draw    (150.86,80.15) -- (150.8,97.04) ;
\draw [shift={(150.86,80.15)}, rotate = 90.2] [color={rgb, 255:red, 0; green, 0; blue, 0 }  ][fill={rgb, 255:red, 0; green, 0; blue, 0 }  ][line width=0.75]      (0, 0) circle [x radius= 3.35, y radius= 3.35]   ;
\draw    (150.8,97.04) -- (150.74,113.92) ;
\draw [shift={(150.8,97.04)}, rotate = 90.2] [color={rgb, 255:red, 0; green, 0; blue, 0 }  ][fill={rgb, 255:red, 0; green, 0; blue, 0 }  ][line width=0.75]      (0, 0) circle [x radius= 3.35, y radius= 3.35]   ;
\draw    (150.74,113.92) -- (150.68,130.81) ;
\draw [shift={(150.74,113.92)}, rotate = 90.2] [color={rgb, 255:red, 0; green, 0; blue, 0 }  ][fill={rgb, 255:red, 0; green, 0; blue, 0 }  ][line width=0.75]      (0, 0) circle [x radius= 3.35, y radius= 3.35]   ;
\draw    (150.68,130.81) -- (150.62,147.7) ;
\draw [shift={(150.68,130.81)}, rotate = 90.2] [color={rgb, 255:red, 0; green, 0; blue, 0 }  ][fill={rgb, 255:red, 0; green, 0; blue, 0 }  ][line width=0.75]      (0, 0) circle [x radius= 3.35, y radius= 3.35]   ;
\draw  [dash pattern={on 0.84pt off 2.51pt}]  (150.62,147.7) -- (150.53,173.03) ;
\draw [shift={(150.62,147.7)}, rotate = 90.2] [color={rgb, 255:red, 0; green, 0; blue, 0 }  ][fill={rgb, 255:red, 0; green, 0; blue, 0 }  ][line width=0.75]      (0, 0) circle [x radius= 3.35, y radius= 3.35]   ;
\draw  [dash pattern={on 0.84pt off 2.51pt}]  (150.56,164.59) -- (150.5,181.47) ;
\draw [shift={(150.5,181.47)}, rotate = 90.2] [color={rgb, 255:red, 0; green, 0; blue, 0 }  ][fill={rgb, 255:red, 0; green, 0; blue, 0 }  ][line width=0.75]      (0, 0) circle [x radius= 3.35, y radius= 3.35]   ;
\draw    (150.5,181.47) -- (150.44,198.36) ;
\draw [shift={(150.5,181.47)}, rotate = 90.2] [color={rgb, 255:red, 0; green, 0; blue, 0 }  ][fill={rgb, 255:red, 0; green, 0; blue, 0 }  ][line width=0.75]      (0, 0) circle [x radius= 3.35, y radius= 3.35]   ;
\draw    (150.44,198.36) -- (150.38,215.25) ;
\draw [shift={(150.44,198.36)}, rotate = 90.2] [color={rgb, 255:red, 0; green, 0; blue, 0 }  ][fill={rgb, 255:red, 0; green, 0; blue, 0 }  ][line width=0.75]      (0, 0) circle [x radius= 3.35, y radius= 3.35]   ;
\draw    (150.38,215.25) -- (150.32,232.13) ;
\draw [shift={(150.32,232.13)}, rotate = 90.2] [color={rgb, 255:red, 0; green, 0; blue, 0 }  ][fill={rgb, 255:red, 0; green, 0; blue, 0 }  ][line width=0.75]      (0, 0) circle [x radius= 3.35, y radius= 3.35]   ;
\draw [shift={(150.38,215.25)}, rotate = 90.2] [color={rgb, 255:red, 0; green, 0; blue, 0 }  ][fill={rgb, 255:red, 0; green, 0; blue, 0 }  ][line width=0.75]      (0, 0) circle [x radius= 3.35, y radius= 3.35]   ;
\draw  [dash pattern={on 0.84pt off 2.51pt}]  (72.9,231.85) -- (234.51,232.02) ;
\draw [color={rgb, 255:red, 248; green, 7; blue, 7 }  ,draw opacity=1 ]   (124.99,113.75) .. controls (127.3,113.28) and (128.68,114.2) .. (129.15,116.51) .. controls (129.62,118.82) and (131.01,119.75) .. (133.32,119.28) .. controls (135.63,118.81) and (137.02,119.74) .. (137.48,122.05) .. controls (137.95,124.36) and (139.34,125.28) .. (141.65,124.81) .. controls (143.96,124.34) and (145.35,125.27) .. (145.81,127.58) .. controls (146.28,129.89) and (147.67,130.81) .. (149.98,130.34) -- (150.68,130.81) -- (150.68,130.81) ;
\draw [shift={(124.99,113.75)}, rotate = 33.59] [color={rgb, 255:red, 248; green, 7; blue, 7 }  ,draw opacity=1 ][fill={rgb, 255:red, 248; green, 7; blue, 7 }  ,fill opacity=1 ][line width=0.75]      (0, 0) circle [x radius= 3.35, y radius= 3.35]   ;
\draw    (127.05,29.67) -- (150.98,46.37) ;
\draw [shift={(127.05,29.67)}, rotate = 34.91] [color={rgb, 255:red, 0; green, 0; blue, 0 }  ][fill={rgb, 255:red, 0; green, 0; blue, 0 }  ][line width=0.75]      (0, 0) circle [x radius= 3.35, y radius= 3.35]   ;
\draw [color={rgb, 255:red, 254; green, 4; blue, 4 }  ,draw opacity=1 ]   (124.99,113.75) .. controls (126.43,115.77) and (126.23,117.58) .. (124.38,119.19) .. controls (122.55,120.79) and (122.41,122.31) .. (123.96,123.76) .. controls (125.49,125.57) and (125.36,127.24) .. (123.59,128.77) .. controls (121.83,130.3) and (121.74,132.09) .. (123.32,134.14) .. controls (124.95,135.71) and (124.92,137.22) .. (123.23,138.69) .. controls (121.56,140.53) and (121.58,142.33) .. (123.29,144.09) .. controls (125.02,145.72) and (125.1,147.26) .. (123.52,148.69) .. controls (122.03,150.82) and (122.2,152.59) .. (124.03,154.01) .. controls (125.89,155.3) and (126.15,156.96) .. (124.8,159) .. controls (123.46,160.8) and (123.8,162.36) .. (125.82,163.69) .. controls (127.87,164.84) and (128.34,166.44) .. (127.24,168.49) .. controls (126.24,170.62) and (126.82,172.12) .. (128.98,172.99) .. controls (131.25,173.91) and (132,175.43) .. (131.23,177.55) .. controls (130.72,179.99) and (131.6,181.42) .. (133.86,181.83) .. controls (136.24,182.29) and (137.25,183.63) .. (136.88,185.85) .. controls (136.65,188.15) and (137.79,189.42) .. (140.3,189.66) .. controls (142.59,189.59) and (143.76,190.71) .. (143.81,193.02) .. controls (144,195.39) and (145.28,196.47) .. (147.67,196.26) -- (150.44,198.36) ;
\draw    (474.75,14.3) -- (474.64,28.19) ;
\draw [shift={(474.77,11.3)}, rotate = 90.44] [fill={rgb, 255:red, 0; green, 0; blue, 0 }  ][line width=0.08]  [draw opacity=0] (8.93,-4.29) -- (0,0) -- (8.93,4.29) -- cycle    ;
\draw    (474.64,28.19) -- (474.51,45.07) ;
\draw [shift={(474.64,28.19)}, rotate = 90.44] [color={rgb, 255:red, 0; green, 0; blue, 0 }  ][fill={rgb, 255:red, 0; green, 0; blue, 0 }  ][line width=0.75]      (0, 0) circle [x radius= 3.35, y radius= 3.35]   ;
\draw [color={rgb, 255:red, 245; green, 5; blue, 5 }  ,draw opacity=1 ]   (474.51,45.07) .. controls (476.16,46.75) and (476.15,48.42) .. (474.47,50.07) .. controls (472.79,51.72) and (472.78,53.39) .. (474.43,55.07) .. controls (476.08,56.75) and (476.07,58.42) .. (474.39,60.07) -- (474.38,61.96) -- (474.38,61.96) ;
\draw [shift={(474.51,45.07)}, rotate = 90.44] [color={rgb, 255:red, 245; green, 5; blue, 5 }  ,draw opacity=1 ][fill={rgb, 255:red, 245; green, 5; blue, 5 }  ,fill opacity=1 ][line width=0.75]      (0, 0) circle [x radius= 3.35, y radius= 3.35]   ;
\draw [color={rgb, 255:red, 240; green, 7; blue, 7 }  ,draw opacity=1 ]   (474.38,61.96) .. controls (476.03,63.64) and (476.02,65.31) .. (474.34,66.96) .. controls (472.66,68.61) and (472.65,70.28) .. (474.3,71.96) .. controls (475.95,73.64) and (475.94,75.31) .. (474.26,76.96) -- (474.25,78.85) -- (474.25,78.85) ;
\draw [shift={(474.38,61.96)}, rotate = 90.44] [color={rgb, 255:red, 240; green, 7; blue, 7 }  ,draw opacity=1 ][fill={rgb, 255:red, 240; green, 7; blue, 7 }  ,fill opacity=1 ][line width=0.75]      (0, 0) circle [x radius= 3.35, y radius= 3.35]   ;
\draw [color={rgb, 255:red, 252; green, 10; blue, 10 }  ,draw opacity=1 ]   (474.25,78.85) .. controls (475.9,80.53) and (475.89,82.2) .. (474.21,83.85) .. controls (472.53,85.5) and (472.52,87.17) .. (474.17,88.85) .. controls (475.82,90.53) and (475.81,92.2) .. (474.13,93.85) -- (474.12,95.73) -- (474.12,95.73) ;
\draw [shift={(474.25,78.85)}, rotate = 90.44] [color={rgb, 255:red, 252; green, 10; blue, 10 }  ,draw opacity=1 ][fill={rgb, 255:red, 252; green, 10; blue, 10 }  ,fill opacity=1 ][line width=0.75]      (0, 0) circle [x radius= 3.35, y radius= 3.35]   ;
\draw    (474.12,95.73) -- (473.98,112.62) ;
\draw [shift={(474.12,95.73)}, rotate = 90.44] [color={rgb, 255:red, 0; green, 0; blue, 0 }  ][fill={rgb, 255:red, 0; green, 0; blue, 0 }  ][line width=0.75]      (0, 0) circle [x radius= 3.35, y radius= 3.35]   ;
\draw    (473.98,112.62) -- (473.85,129.51) ;
\draw [shift={(473.98,112.62)}, rotate = 90.44] [color={rgb, 255:red, 0; green, 0; blue, 0 }  ][fill={rgb, 255:red, 0; green, 0; blue, 0 }  ][line width=0.75]      (0, 0) circle [x radius= 3.35, y radius= 3.35]   ;
\draw    (473.85,129.51) -- (473.72,146.4) ;
\draw [shift={(473.85,129.51)}, rotate = 90.44] [color={rgb, 255:red, 0; green, 0; blue, 0 }  ][fill={rgb, 255:red, 0; green, 0; blue, 0 }  ][line width=0.75]      (0, 0) circle [x radius= 3.35, y radius= 3.35]   ;
\draw  [dash pattern={on 0.84pt off 2.51pt}]  (473.72,146.4) -- (473.53,171.73) ;
\draw [shift={(473.72,146.4)}, rotate = 90.44] [color={rgb, 255:red, 0; green, 0; blue, 0 }  ][fill={rgb, 255:red, 0; green, 0; blue, 0 }  ][line width=0.75]      (0, 0) circle [x radius= 3.35, y radius= 3.35]   ;
\draw  [dash pattern={on 0.84pt off 2.51pt}]  (473.59,163.28) -- (473.46,180.17) ;
\draw    (473.46,180.17) -- (473.33,197.06) ;
\draw [shift={(473.46,180.17)}, rotate = 90.44] [color={rgb, 255:red, 0; green, 0; blue, 0 }  ][fill={rgb, 255:red, 0; green, 0; blue, 0 }  ][line width=0.75]      (0, 0) circle [x radius= 3.35, y radius= 3.35]   ;
\draw    (473.33,197.06) -- (473.2,213.94) ;
\draw [shift={(473.33,197.06)}, rotate = 90.44] [color={rgb, 255:red, 0; green, 0; blue, 0 }  ][fill={rgb, 255:red, 0; green, 0; blue, 0 }  ][line width=0.75]      (0, 0) circle [x radius= 3.35, y radius= 3.35]   ;
\draw    (473.2,213.94) -- (473.07,230.83) ;
\draw [shift={(473.07,230.83)}, rotate = 90.44] [color={rgb, 255:red, 0; green, 0; blue, 0 }  ][fill={rgb, 255:red, 0; green, 0; blue, 0 }  ][line width=0.75]      (0, 0) circle [x radius= 3.35, y radius= 3.35]   ;
\draw [shift={(473.2,213.94)}, rotate = 90.44] [color={rgb, 255:red, 0; green, 0; blue, 0 }  ][fill={rgb, 255:red, 0; green, 0; blue, 0 }  ][line width=0.75]      (0, 0) circle [x radius= 3.35, y radius= 3.35]   ;
\draw  [dash pattern={on 0.84pt off 2.51pt}]  (395.65,230.55) -- (557.26,230.71) ;
\draw [color={rgb, 255:red, 3; green, 0; blue, 0 }  ,draw opacity=1 ]   (448.23,112.45) -- (473.85,129.51) ;
\draw [shift={(448.23,112.45)}, rotate = 33.66] [color={rgb, 255:red, 3; green, 0; blue, 0 }  ,draw opacity=1 ][fill={rgb, 255:red, 3; green, 0; blue, 0 }  ,fill opacity=1 ][line width=0.75]      (0, 0) circle [x radius= 3.35, y radius= 3.35]   ;
\draw [color={rgb, 255:red, 252; green, 3; blue, 3 }  ,draw opacity=1 ]   (450.64,28.37) .. controls (452.97,27.96) and (454.33,28.92) .. (454.74,31.24) .. controls (455.15,33.56) and (456.52,34.51) .. (458.84,34.1) .. controls (461.16,33.69) and (462.52,34.65) .. (462.93,36.97) .. controls (463.34,39.29) and (464.71,40.25) .. (467.03,39.84) .. controls (469.35,39.43) and (470.71,40.38) .. (471.12,42.7) -- (474.51,45.07) -- (474.51,45.07) ;
\draw [shift={(450.64,28.37)}, rotate = 34.99] [color={rgb, 255:red, 252; green, 3; blue, 3 }  ,draw opacity=1 ][fill={rgb, 255:red, 252; green, 3; blue, 3 }  ,fill opacity=1 ][line width=0.75]      (0, 0) circle [x radius= 3.35, y radius= 3.35]   ;
\draw [color={rgb, 255:red, 7; green, 1; blue, 1 }  ,draw opacity=1 ]   (448.23,112.45) .. controls (442.49,158.81) and (449.67,179.69) .. (473.33,197.06) ;
\draw [color={rgb, 255:red, 255; green, 0; blue, 0 }  ,draw opacity=1 ]   (450.64,28.37) .. controls (449.33,30.57) and (447.68,31.22) .. (445.67,30.31) .. controls (443.48,29.65) and (442,30.61) .. (441.23,33.19) .. controls (441.2,35.41) and (440.06,36.46) .. (437.83,36.34) .. controls (435.49,36.49) and (434.48,37.74) .. (434.8,40.08) .. controls (435.29,42.35) and (434.4,43.77) .. (432.12,44.36) .. controls (429.82,45.19) and (429.05,46.78) .. (429.81,49.14) .. controls (430.68,51.43) and (430.13,52.87) .. (428.14,53.46) .. controls (425.95,54.8) and (425.39,56.65) .. (426.45,59.02) .. controls (427.74,60.71) and (427.36,62.35) .. (425.3,63.94) .. controls (423.39,64.99) and (423.14,66.36) .. (424.54,68.04) .. controls (425.87,70.37) and (425.62,72.13) .. (423.79,73.34) .. controls (421.97,74.62) and (421.79,76.45) .. (423.26,78.82) .. controls (424.82,80.41) and (424.73,81.91) .. (422.98,83.32) .. controls (421.26,84.79) and (421.22,86.32) .. (422.85,87.9) .. controls (424.5,90.21) and (424.51,92.15) .. (422.87,93.72) .. controls (421.24,95.35) and (421.3,96.92) .. (423.03,98.44) .. controls (424.78,99.93) and (424.87,101.51) .. (423.31,103.2) .. controls (421.78,104.94) and (421.92,106.54) .. (423.73,107.99) .. controls (425.55,109.4) and (425.73,111) .. (424.26,112.78) .. controls (422.93,115.41) and (423.21,117.41) .. (425.1,118.77) .. controls (427,120.09) and (427.27,121.68) .. (425.9,123.55) .. controls (424.57,125.45) and (424.87,127.03) .. (426.82,128.29) .. controls (428.79,129.5) and (429.13,131.07) .. (427.86,132.98) .. controls (426.61,134.93) and (426.99,136.48) .. (429,137.62) .. controls (431.02,138.71) and (431.44,140.23) .. (430.26,142.18) .. controls (429.34,144.92) and (429.91,146.78) .. (431.98,147.76) .. controls (434.04,148.68) and (434.54,150.13) .. (433.47,152.11) .. controls (432.43,154.12) and (432.96,155.53) .. (435.06,156.34) .. controls (437.43,157.75) and (438.14,159.46) .. (437.19,161.45) .. controls (436.28,163.47) and (437.04,165.09) .. (439.48,166.32) .. controls (441.59,166.81) and (442.24,168.05) .. (441.41,170.04) .. controls (440.98,172.66) and (441.82,174.12) .. (443.95,174.43) .. controls (446.4,175.16) and (447.3,176.52) .. (446.64,178.51) .. controls (446.43,181.05) and (447.56,182.54) .. (450.03,182.97) .. controls (452.06,182.78) and (453.06,183.89) .. (453.01,186.29) .. controls (453.1,188.72) and (454.35,189.87) .. (456.74,189.75) .. controls (459.01,189.41) and (460.32,190.36) .. (460.65,192.61) .. controls (461.69,195.12) and (463.28,195.96) .. (465.42,195.11) .. controls (467.34,194.04) and (469,194.55) .. (470.39,196.64) -- (473.33,197.06) ;
\draw    (127.05,29.67) .. controls (94.01,35.06) and (85.08,194.85) .. (150.44,198.36) ;
\draw   (286,109.15) -- (328,109.15) -- (328,102) -- (356,116.3) -- (328,130.6) -- (328,123.45) -- (286,123.45) -- cycle ;

\draw (569.79,220.69) node [anchor=north west][inner sep=0.75pt]   [align=left] {$n_0$};
\draw (482.47,188.11) node [anchor=north west][inner sep=0.75pt]  [xslant=-0.06] [align=left] {w};
\draw (483.97,120.98) node [anchor=north west][inner sep=0.75pt]   [align=left] {v};
\draw (426.46,103.3) node [anchor=north west][inner sep=0.75pt]   [align=left] {u};
\draw (483.66,86.6) node [anchor=north west][inner sep=0.75pt]   [align=left] {$v_0$};
\draw (486.32,36.42) node [anchor=north west][inner sep=0.75pt]   [align=left] {$v_1$};
\draw (415.1,19.19) node [anchor=north west][inner sep=0.75pt]   [align=left] {$u_1$};
\draw (247,221.99) node [anchor=north west][inner sep=0.75pt]   [align=left] {$n_0$};
\draw (159.55,189.41) node [anchor=north west][inner sep=0.75pt]  [xslant=-0.06] [align=left] {w};
\draw (160.77,122.28) node [anchor=north west][inner sep=0.75pt]   [align=left] {v};
\draw (103.19,104.6) node [anchor=north west][inner sep=0.75pt]   [align=left] {u};
\draw (160.32,87.9) node [anchor=north west][inner sep=0.75pt]   [align=left] {$v_0$};
\draw (162.77,37.72) node [anchor=north west][inner sep=0.75pt]   [align=left] {$v_1$};
\draw (91.48,20.49) node [anchor=north west][inner sep=0.75pt]   [align=left] {$u_1$};

\end{tikzpicture}
    \caption{On the left side, in red, are the edges $ wu $ and $ vu $ that were originally part of $ C $. On the right side, in red, the process mentioned in the text has already been carried out. Now, the edges that are part of $ C $ are $ wu_1 $, $ u_1v_1 $ and $v_1Tv_0$.
}
    \label{figfinal1}
\end{figure}
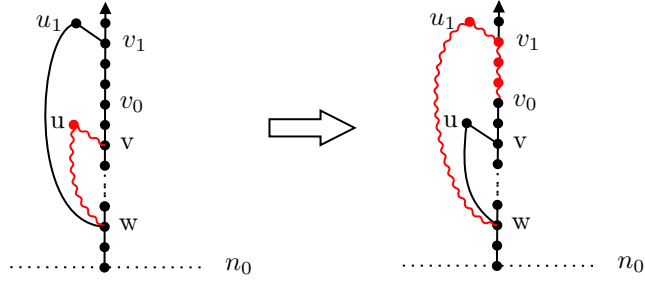

\textbf{Case 1:} $u\notin L_j$.

Then $u$ has no successors in $T$, and hence we have $v_j^{n_0}\leq w < v< u$. Let $v_1\in A\cap \lfloor v\rfloor$ be such that $v_1\in A_w$. Then, there exists $u_1\in B\setminus B^{\infty}$ such that $v_1, w\in N(u_1)$. Then remove the edges $vu, wu$ from $C$ and add the path $wu_1v_1Tv_0$ in $C$, where $v_0$ is the first vertex of $A$ above $v$ in the tree order (see Figure \ref{figfinal1}).

 Let $v_2\in A\cap \lfloor v_1\rfloor$ be such that $v_2\in A_v$. Then, there exists $u_2\in B\setminus B^{\infty}$ such that $v_2, v\in N(u_2)$ (see Figure \ref{figfinal2}). Then, add the path $vu_2v_2Tv_3$ into $C$, where $v_3$ is the first vertex of $A$ above $v_1$ in the tree order. Repeat the process inductively for all $v_i$'s obtained. Thus, we can cover all the vertices of $A$ that are in $\lfloor v\rfloor \cup \lceil v\rceil$.

\begin{figure}[ht]
    \centering

\tikzset{every picture/.style={line width=0.75pt}} 

\begin{tikzpicture}[x=0.55pt,y=0.55pt,yscale=-0.75,xscale=0.75]

\draw    (184.63,8.92) -- (184.51,26.3) ;
\draw [shift={(184.65,5.92)}, rotate = 90.4] [fill={rgb, 255:red, 0; green, 0; blue, 0 }  ][line width=0.08]  [draw opacity=0] (8.93,-4.29) -- (0,0) -- (8.93,4.29) -- cycle    ;
\draw    (184.51,26.3) -- (184.36,46.69) ;
\draw [shift={(184.51,26.3)}, rotate = 90.4] [color={rgb, 255:red, 0; green, 0; blue, 0 }  ][fill={rgb, 255:red, 0; green, 0; blue, 0 }  ][line width=0.75]      (0, 0) circle [x radius= 3.35, y radius= 3.35]   ;
\draw    (184.36,46.69) -- (184.12,66.53) ;
\draw [shift={(184.36,46.69)}, rotate = 90.69] [color={rgb, 255:red, 0; green, 0; blue, 0 }  ][fill={rgb, 255:red, 0; green, 0; blue, 0 }  ][line width=0.75]      (0, 0) circle [x radius= 3.35, y radius= 3.35]   ;
\draw    (184.12,66.53) -- (183.99,85.78) ;
\draw [shift={(183.99,85.78)}, rotate = 90.4] [color={rgb, 255:red, 0; green, 0; blue, 0 }  ][fill={rgb, 255:red, 0; green, 0; blue, 0 }  ][line width=0.75]      (0, 0) circle [x radius= 3.35, y radius= 3.35]   ;
\draw [shift={(184.12,66.53)}, rotate = 90.4] [color={rgb, 255:red, 0; green, 0; blue, 0 }  ][fill={rgb, 255:red, 0; green, 0; blue, 0 }  ][line width=0.75]      (0, 0) circle [x radius= 3.35, y radius= 3.35]   ;
\draw  [dash pattern={on 0.84pt off 2.51pt}]  (183.99,85.78) -- (183.85,105.04) ;
\draw    (183.85,105.04) -- (183.72,124.29) ;
\draw [shift={(183.85,105.04)}, rotate = 90.4] [color={rgb, 255:red, 0; green, 0; blue, 0 }  ][fill={rgb, 255:red, 0; green, 0; blue, 0 }  ][line width=0.75]      (0, 0) circle [x radius= 3.35, y radius= 3.35]   ;
\draw    (183.72,124.29) -- (183.58,143.55) ;
\draw [shift={(183.72,124.29)}, rotate = 90.4] [color={rgb, 255:red, 0; green, 0; blue, 0 }  ][fill={rgb, 255:red, 0; green, 0; blue, 0 }  ][line width=0.75]      (0, 0) circle [x radius= 3.35, y radius= 3.35]   ;
\draw    (183.58,143.55) -- (183.45,162.81) ;
\draw [shift={(183.58,143.55)}, rotate = 90.4] [color={rgb, 255:red, 0; green, 0; blue, 0 }  ][fill={rgb, 255:red, 0; green, 0; blue, 0 }  ][line width=0.75]      (0, 0) circle [x radius= 3.35, y radius= 3.35]   ;
\draw [color={rgb, 255:red, 252; green, 4; blue, 4 }  ,draw opacity=1 ]   (183.45,162.81) .. controls (185.1,164.49) and (185.09,166.16) .. (183.41,167.81) .. controls (181.73,169.46) and (181.72,171.12) .. (183.37,172.8) .. controls (185.02,174.47) and (185.01,176.14) .. (183.34,177.8) -- (183.31,182.06) -- (183.31,182.06) ;
\draw [shift={(183.45,162.81)}, rotate = 90.4] [color={rgb, 255:red, 252; green, 4; blue, 4 }  ,draw opacity=1 ][fill={rgb, 255:red, 252; green, 4; blue, 4 }  ,fill opacity=1 ][line width=0.75]      (0, 0) circle [x radius= 3.35, y radius= 3.35]   ;
\draw [color={rgb, 255:red, 248; green, 2; blue, 2 }  ,draw opacity=1 ] [dash pattern={on 0.84pt off 2.51pt}]  (183.31,182.06) -- (183.17,201.32) ;
\draw [shift={(183.31,182.06)}, rotate = 90.4] [color={rgb, 255:red, 248; green, 2; blue, 2 }  ,draw opacity=1 ][fill={rgb, 255:red, 248; green, 2; blue, 2 }  ,fill opacity=1 ][line width=0.75]      (0, 0) circle [x radius= 3.35, y radius= 3.35]   ;
\draw [color={rgb, 255:red, 255; green, 6; blue, 6 }  ,draw opacity=1 ]   (183.17,201.32) .. controls (184.83,202.99) and (184.82,204.66) .. (183.14,206.32) .. controls (181.46,207.97) and (181.45,209.64) .. (183.1,211.32) .. controls (184.75,212.99) and (184.74,214.66) .. (183.07,216.32) -- (183.04,220.57) -- (183.04,220.57) ;
\draw [shift={(183.17,201.32)}, rotate = 90.4] [color={rgb, 255:red, 255; green, 6; blue, 6 }  ,draw opacity=1 ][fill={rgb, 255:red, 255; green, 6; blue, 6 }  ,fill opacity=1 ][line width=0.75]      (0, 0) circle [x radius= 3.35, y radius= 3.35]   ;
\draw    (183.04,220.57) -- (182.9,239.83) ;
\draw [shift={(183.04,220.57)}, rotate = 90.4] [color={rgb, 255:red, 0; green, 0; blue, 0 }  ][fill={rgb, 255:red, 0; green, 0; blue, 0 }  ][line width=0.75]      (0, 0) circle [x radius= 3.35, y radius= 3.35]   ;
\draw    (182.9,239.83) -- (182.77,259.08) ;
\draw [shift={(182.77,259.08)}, rotate = 90.4] [color={rgb, 255:red, 0; green, 0; blue, 0 }  ][fill={rgb, 255:red, 0; green, 0; blue, 0 }  ][line width=0.75]      (0, 0) circle [x radius= 3.35, y radius= 3.35]   ;
\draw [shift={(182.9,239.83)}, rotate = 90.4] [color={rgb, 255:red, 0; green, 0; blue, 0 }  ][fill={rgb, 255:red, 0; green, 0; blue, 0 }  ][line width=0.75]      (0, 0) circle [x radius= 3.35, y radius= 3.35]   ;
\draw    (184.12,66.53) -- (211.48,47.64) ;
\draw [shift={(211.48,47.64)}, rotate = 325.37] [color={rgb, 255:red, 0; green, 0; blue, 0 }  ][fill={rgb, 255:red, 0; green, 0; blue, 0 }  ][line width=0.75]      (0, 0) circle [x radius= 3.35, y radius= 3.35]   ;
\draw [shift={(184.12,66.53)}, rotate = 325.37] [color={rgb, 255:red, 0; green, 0; blue, 0 }  ][fill={rgb, 255:red, 0; green, 0; blue, 0 }  ][line width=0.75]      (0, 0) circle [x radius= 3.35, y radius= 3.35]   ;
\draw [color={rgb, 255:red, 255; green, 47; blue, 47 }  ,draw opacity=1 ]   (149.85,266.48) .. controls (148.17,264.79) and (148.15,263.01) .. (149.79,261.16) .. controls (151.43,259.52) and (151.4,257.87) .. (149.7,256.22) .. controls (148.01,254.5) and (147.98,252.9) .. (149.62,251.43) .. controls (151.25,249.54) and (151.22,247.84) .. (149.53,246.34) .. controls (147.84,244.39) and (147.82,242.61) .. (149.47,240.98) .. controls (151.12,239.33) and (151.11,237.68) .. (149.43,236.02) .. controls (147.76,234.33) and (147.76,232.63) .. (149.43,230.91) .. controls (151.1,229.62) and (151.11,228.09) .. (149.46,226.33) .. controls (147.82,224.53) and (147.85,222.75) .. (149.54,220.99) .. controls (151.24,219.68) and (151.28,218.09) .. (149.66,216.24) .. controls (148.05,214.35) and (148.11,212.75) .. (149.83,211.44) .. controls (151.58,209.68) and (151.66,207.84) .. (150.09,205.91) .. controls (148.51,204.43) and (148.61,202.81) .. (150.39,201.06) .. controls (152.15,199.8) and (152.27,198.18) .. (150.75,196.21) .. controls (149.26,194.21) and (149.41,192.6) .. (151.2,191.37) .. controls (153.05,189.72) and (153.25,187.89) .. (151.8,185.88) .. controls (150.31,184.31) and (150.52,182.72) .. (152.42,181.13) .. controls (154.27,180.04) and (154.51,178.48) .. (153.14,176.45) .. controls (151.8,174.4) and (152.11,172.65) .. (154.08,171.2) .. controls (155.97,170.25) and (156.28,168.75) .. (155.02,166.72) .. controls (153.79,164.65) and (154.19,162.99) .. (156.23,161.74) .. controls (158.27,160.58) and (158.72,158.98) .. (157.59,156.93) .. controls (156.51,154.86) and (157.02,153.32) .. (159.13,152.33) .. controls (161.37,151.1) and (162.01,149.47) .. (161.06,147.43) -- (162.72,143.83) ;
\draw [color={rgb, 255:red, 253; green, 7; blue, 7 }  ,draw opacity=1 ]   (183.45,162.81) .. controls (181.09,162.91) and (179.86,161.78) .. (179.76,159.43) .. controls (179.65,157.08) and (178.42,155.95) .. (176.07,156.05) .. controls (173.72,156.16) and (172.49,155.03) .. (172.38,152.68) .. controls (172.27,150.33) and (171.04,149.2) .. (168.69,149.3) .. controls (166.34,149.4) and (165.11,148.27) .. (165.01,145.92) -- (162.72,143.83) -- (162.72,143.83) ;
\draw [shift={(162.72,143.83)}, rotate = 222.48] [color={rgb, 255:red, 253; green, 7; blue, 7 }  ,draw opacity=1 ][fill={rgb, 255:red, 253; green, 7; blue, 7 }  ,fill opacity=1 ][line width=0.75]      (0, 0) circle [x radius= 3.35, y radius= 3.35]   ;
\draw [shift={(183.45,162.81)}, rotate = 222.48] [color={rgb, 255:red, 253; green, 7; blue, 7 }  ,draw opacity=1 ][fill={rgb, 255:red, 253; green, 7; blue, 7 }  ,fill opacity=1 ][line width=0.75]      (0, 0) circle [x radius= 3.35, y radius= 3.35]   ;
\draw    (168.44,240.8) -- (182.77,259.08) ;
\draw [shift={(182.77,259.08)}, rotate = 51.92] [color={rgb, 255:red, 0; green, 0; blue, 0 }  ][fill={rgb, 255:red, 0; green, 0; blue, 0 }  ][line width=0.75]      (0, 0) circle [x radius= 3.35, y radius= 3.35]   ;
\draw [shift={(168.44,240.8)}, rotate = 51.92] [color={rgb, 255:red, 0; green, 0; blue, 0 }  ][fill={rgb, 255:red, 0; green, 0; blue, 0 }  ][line width=0.75]      (0, 0) circle [x radius= 3.35, y radius= 3.35]   ;
\draw    (182.77,259.08) .. controls (223.43,197.72) and (235.28,83.18) .. (211.48,47.64) ;
\draw    (429.64,14.39) -- (429.83,30.43) ;
\draw [shift={(429.61,11.39)}, rotate = 89.33] [fill={rgb, 255:red, 0; green, 0; blue, 0 }  ][line width=0.08]  [draw opacity=0] (8.93,-4.29) -- (0,0) -- (8.93,4.29) -- cycle    ;
\draw    (429.83,30.43) -- (430.05,49.46) ;
\draw [shift={(429.83,30.43)}, rotate = 89.33] [color={rgb, 255:red, 0; green, 0; blue, 0 }  ][fill={rgb, 255:red, 0; green, 0; blue, 0 }  ][line width=0.75]      (0, 0) circle [x radius= 3.35, y radius= 3.35]   ;
\draw    (430.05,49.46) -- (429.81,69.3) ;
\draw [shift={(430.05,49.46)}, rotate = 90.69] [color={rgb, 255:red, 0; green, 0; blue, 0 }  ][fill={rgb, 255:red, 0; green, 0; blue, 0 }  ][line width=0.75]      (0, 0) circle [x radius= 3.35, y radius= 3.35]   ;
\draw [color={rgb, 255:red, 255; green, 0; blue, 0 }  ,draw opacity=1 ]   (429.81,69.3) .. controls (431.47,70.97) and (431.46,72.64) .. (429.78,74.3) .. controls (428.1,75.95) and (428.09,77.62) .. (429.74,79.3) .. controls (431.39,80.97) and (431.38,82.64) .. (429.71,84.3) -- (429.68,88.55) -- (429.68,88.55) ;
\draw [shift={(429.68,88.55)}, rotate = 90.4] [color={rgb, 255:red, 255; green, 0; blue, 0 }  ,draw opacity=1 ][fill={rgb, 255:red, 255; green, 0; blue, 0 }  ,fill opacity=1 ][line width=0.75]      (0, 0) circle [x radius= 3.35, y radius= 3.35]   ;
\draw [shift={(429.81,69.3)}, rotate = 90.4] [color={rgb, 255:red, 255; green, 0; blue, 0 }  ,draw opacity=1 ][fill={rgb, 255:red, 255; green, 0; blue, 0 }  ,fill opacity=1 ][line width=0.75]      (0, 0) circle [x radius= 3.35, y radius= 3.35]   ;
\draw [color={rgb, 255:red, 255; green, 0; blue, 0 }  ,draw opacity=1 ] [dash pattern={on 0.84pt off 2.51pt}]  (429.68,88.55) -- (429.54,107.81) ;
\draw [color={rgb, 255:red, 255; green, 0; blue, 0 }  ,draw opacity=1 ]   (429.54,107.81) .. controls (431.2,109.48) and (431.19,111.15) .. (429.51,112.81) .. controls (427.83,114.46) and (427.82,116.13) .. (429.47,117.81) .. controls (431.12,119.48) and (431.11,121.15) .. (429.44,122.81) -- (429.41,127.06) -- (429.41,127.06) ;
\draw [shift={(429.54,107.81)}, rotate = 90.4] [color={rgb, 255:red, 255; green, 0; blue, 0 }  ,draw opacity=1 ][fill={rgb, 255:red, 255; green, 0; blue, 0 }  ,fill opacity=1 ][line width=0.75]      (0, 0) circle [x radius= 3.35, y radius= 3.35]   ;
\draw    (429.41,127.06) -- (429.27,146.32) ;
\draw [shift={(429.41,127.06)}, rotate = 90.4] [color={rgb, 255:red, 0; green, 0; blue, 0 }  ][fill={rgb, 255:red, 0; green, 0; blue, 0 }  ][line width=0.75]      (0, 0) circle [x radius= 3.35, y radius= 3.35]   ;
\draw    (429.27,146.32) -- (429.13,165.57) ;
\draw [shift={(429.27,146.32)}, rotate = 90.4] [color={rgb, 255:red, 0; green, 0; blue, 0 }  ][fill={rgb, 255:red, 0; green, 0; blue, 0 }  ][line width=0.75]      (0, 0) circle [x radius= 3.35, y radius= 3.35]   ;
\draw [color={rgb, 255:red, 252; green, 4; blue, 4 }  ,draw opacity=1 ]   (429.13,165.57) .. controls (430.79,167.25) and (430.78,168.92) .. (429.1,170.57) .. controls (427.42,172.22) and (427.41,173.89) .. (429.06,175.57) .. controls (430.71,177.24) and (430.7,178.91) .. (429.03,180.57) -- (429,184.83) -- (429,184.83) ;
\draw [shift={(429.13,165.57)}, rotate = 90.4] [color={rgb, 255:red, 252; green, 4; blue, 4 }  ,draw opacity=1 ][fill={rgb, 255:red, 252; green, 4; blue, 4 }  ,fill opacity=1 ][line width=0.75]      (0, 0) circle [x radius= 3.35, y radius= 3.35]   ;
\draw [color={rgb, 255:red, 248; green, 2; blue, 2 }  ,draw opacity=1 ] [dash pattern={on 0.84pt off 2.51pt}]  (429,184.83) -- (428.86,204.08) ;
\draw [shift={(429,184.83)}, rotate = 90.4] [color={rgb, 255:red, 248; green, 2; blue, 2 }  ,draw opacity=1 ][fill={rgb, 255:red, 248; green, 2; blue, 2 }  ,fill opacity=1 ][line width=0.75]      (0, 0) circle [x radius= 3.35, y radius= 3.35]   ;
\draw [color={rgb, 255:red, 255; green, 6; blue, 6 }  ,draw opacity=1 ]   (428.86,204.08) .. controls (430.52,205.76) and (430.51,207.43) .. (428.83,209.08) .. controls (427.15,210.73) and (427.14,212.4) .. (428.79,214.08) .. controls (430.44,215.75) and (430.43,217.42) .. (428.76,219.08) -- (428.73,223.34) -- (428.73,223.34) ;
\draw [shift={(428.86,204.08)}, rotate = 90.4] [color={rgb, 255:red, 255; green, 6; blue, 6 }  ,draw opacity=1 ][fill={rgb, 255:red, 255; green, 6; blue, 6 }  ,fill opacity=1 ][line width=0.75]      (0, 0) circle [x radius= 3.35, y radius= 3.35]   ;
\draw    (428.73,223.34) -- (428.59,242.59) ;
\draw [shift={(428.73,223.34)}, rotate = 90.4] [color={rgb, 255:red, 0; green, 0; blue, 0 }  ][fill={rgb, 255:red, 0; green, 0; blue, 0 }  ][line width=0.75]      (0, 0) circle [x radius= 3.35, y radius= 3.35]   ;
\draw    (428.59,242.59) -- (428.45,261.85) ;
\draw [shift={(428.45,261.85)}, rotate = 90.4] [color={rgb, 255:red, 0; green, 0; blue, 0 }  ][fill={rgb, 255:red, 0; green, 0; blue, 0 }  ][line width=0.75]      (0, 0) circle [x radius= 3.35, y radius= 3.35]   ;
\draw [shift={(428.59,242.59)}, rotate = 90.4] [color={rgb, 255:red, 0; green, 0; blue, 0 }  ][fill={rgb, 255:red, 0; green, 0; blue, 0 }  ][line width=0.75]      (0, 0) circle [x radius= 3.35, y radius= 3.35]   ;
\draw [color={rgb, 255:red, 249; green, 3; blue, 3 }  ,draw opacity=1 ]   (429.81,69.3) .. controls (430.24,66.98) and (431.61,66.03) .. (433.93,66.45) .. controls (436.25,66.88) and (437.62,65.93) .. (438.04,63.61) .. controls (438.47,61.29) and (439.84,60.34) .. (442.16,60.77) .. controls (444.48,61.2) and (445.85,60.25) .. (446.27,57.93) .. controls (446.69,55.61) and (448.06,54.66) .. (450.38,55.09) .. controls (452.7,55.52) and (454.07,54.57) .. (454.5,52.25) -- (457.17,50.4) -- (457.17,50.4) ;
\draw [shift={(457.17,50.4)}, rotate = 325.37] [color={rgb, 255:red, 249; green, 3; blue, 3 }  ,draw opacity=1 ][fill={rgb, 255:red, 249; green, 3; blue, 3 }  ,fill opacity=1 ][line width=0.75]      (0, 0) circle [x radius= 3.35, y radius= 3.35]   ;
\draw [shift={(429.81,69.3)}, rotate = 325.37] [color={rgb, 255:red, 249; green, 3; blue, 3 }  ,draw opacity=1 ][fill={rgb, 255:red, 249; green, 3; blue, 3 }  ,fill opacity=1 ][line width=0.75]      (0, 0) circle [x radius= 3.35, y radius= 3.35]   ;
\draw [color={rgb, 255:red, 255; green, 47; blue, 47 }  ,draw opacity=1 ]   (395.54,269.24) .. controls (393.86,267.55) and (393.84,265.78) .. (395.48,263.93) .. controls (397.12,262.29) and (397.09,260.64) .. (395.39,258.99) .. controls (393.7,257.27) and (393.67,255.67) .. (395.31,254.2) .. controls (396.94,252.31) and (396.91,250.61) .. (395.22,249.1) .. controls (393.53,247.16) and (393.51,245.38) .. (395.16,243.75) .. controls (396.81,242.1) and (396.8,240.44) .. (395.12,238.79) .. controls (393.45,237.1) and (393.45,235.39) .. (395.12,233.68) .. controls (396.79,232.39) and (396.8,230.86) .. (395.15,229.09) .. controls (393.51,227.3) and (393.54,225.52) .. (395.23,223.75) .. controls (396.93,222.44) and (396.97,220.86) .. (395.35,219) .. controls (393.74,217.12) and (393.8,215.52) .. (395.52,214.21) .. controls (397.27,212.45) and (397.35,210.61) .. (395.78,208.68) .. controls (394.2,207.19) and (394.3,205.58) .. (396.08,203.83) .. controls (397.84,202.56) and (397.96,200.94) .. (396.44,198.97) .. controls (394.95,196.98) and (395.09,195.37) .. (396.88,194.14) .. controls (398.73,192.49) and (398.94,190.66) .. (397.49,188.65) .. controls (396,187.08) and (396.21,185.49) .. (398.11,183.9) .. controls (399.96,182.81) and (400.2,181.25) .. (398.83,179.22) .. controls (397.49,177.17) and (397.8,175.42) .. (399.77,173.97) .. controls (401.66,173.02) and (401.97,171.52) .. (400.71,169.48) .. controls (399.48,167.41) and (399.88,165.75) .. (401.92,164.5) .. controls (403.96,163.34) and (404.41,161.74) .. (403.28,159.7) .. controls (402.2,157.62) and (402.71,156.09) .. (404.81,155.1) .. controls (407.06,153.87) and (407.7,152.24) .. (406.75,150.19) -- (408.41,146.6) ;
\draw [color={rgb, 255:red, 253; green, 7; blue, 7 }  ,draw opacity=1 ]   (429.13,165.57) .. controls (426.78,165.68) and (425.55,164.55) .. (425.45,162.2) .. controls (425.34,159.85) and (424.11,158.72) .. (421.76,158.82) .. controls (419.41,158.92) and (418.18,157.79) .. (418.07,155.44) .. controls (417.96,153.09) and (416.73,151.96) .. (414.38,152.07) .. controls (412.03,152.17) and (410.8,151.04) .. (410.7,148.69) -- (408.41,146.6) -- (408.41,146.6) ;
\draw [shift={(408.41,146.6)}, rotate = 222.48] [color={rgb, 255:red, 253; green, 7; blue, 7 }  ,draw opacity=1 ][fill={rgb, 255:red, 253; green, 7; blue, 7 }  ,fill opacity=1 ][line width=0.75]      (0, 0) circle [x radius= 3.35, y radius= 3.35]   ;
\draw [shift={(429.13,165.57)}, rotate = 222.48] [color={rgb, 255:red, 253; green, 7; blue, 7 }  ,draw opacity=1 ][fill={rgb, 255:red, 253; green, 7; blue, 7 }  ,fill opacity=1 ][line width=0.75]      (0, 0) circle [x radius= 3.35, y radius= 3.35]   ;
\draw    (414.13,243.56) -- (428.45,261.85) ;
\draw [shift={(428.45,261.85)}, rotate = 51.92] [color={rgb, 255:red, 0; green, 0; blue, 0 }  ][fill={rgb, 255:red, 0; green, 0; blue, 0 }  ][line width=0.75]      (0, 0) circle [x radius= 3.35, y radius= 3.35]   ;
\draw [shift={(414.13,243.56)}, rotate = 51.92] [color={rgb, 255:red, 0; green, 0; blue, 0 }  ][fill={rgb, 255:red, 0; green, 0; blue, 0 }  ][line width=0.75]      (0, 0) circle [x radius= 3.35, y radius= 3.35]   ;
\draw [color={rgb, 255:red, 252; green, 16; blue, 16 }  ,draw opacity=1 ]   (428.45,261.85) .. controls (428.25,259.09) and (429.25,257.52) .. (431.45,257.15) .. controls (433.66,256.71) and (434.44,255.41) .. (433.77,253.25) .. controls (433.44,250.45) and (434.37,248.78) .. (436.57,248.23) .. controls (438.77,247.62) and (439.49,246.24) .. (438.73,244.08) .. controls (437.94,241.94) and (438.63,240.53) .. (440.82,239.84) .. controls (443.01,239.11) and (443.68,237.66) .. (442.84,235.51) .. controls (442.29,232.64) and (443.09,230.8) .. (445.26,229.97) .. controls (447.43,229.11) and (448.05,227.61) .. (447.12,225.46) .. controls (446.17,223.33) and (446.76,221.8) .. (448.91,220.87) .. controls (451.05,219.9) and (451.62,218.34) .. (450.62,216.21) .. controls (449.6,214.1) and (450.15,212.52) .. (452.26,211.49) .. controls (454.37,210.43) and (454.9,208.84) .. (453.83,206.72) .. controls (452.74,204.63) and (453.24,203.02) .. (455.32,201.91) .. controls (457.4,200.76) and (457.88,199.14) .. (456.75,197.05) .. controls (455.6,194.98) and (456.04,193.35) .. (458.09,192.16) .. controls (460.14,190.94) and (460.56,189.3) .. (459.37,187.24) .. controls (458.15,185.21) and (458.55,183.56) .. (460.56,182.31) .. controls (462.57,181.02) and (462.85,179.78) .. (461.41,178.59) .. controls (460.14,176.59) and (460.5,174.94) .. (462.48,173.63) .. controls (464.45,172.3) and (464.78,170.65) .. (463.47,168.66) .. controls (462.14,166.71) and (462.45,165.05) .. (464.39,163.7) .. controls (466.32,162.33) and (466.6,160.68) .. (465.23,158.75) .. controls (463.84,156.84) and (464.1,155.2) .. (465.99,153.81) .. controls (467.88,152.4) and (468.1,150.76) .. (466.67,148.89) .. controls (465.23,147.05) and (465.43,145.42) .. (467.28,144) .. controls (469.12,142.56) and (469.3,140.94) .. (467.81,139.15) .. controls (466.38,136.59) and (466.57,134.58) .. (468.37,133.13) .. controls (470.16,131.67) and (470.27,130.09) .. (468.72,128.38) .. controls (467.15,126.71) and (467.25,125.14) .. (469,123.68) .. controls (470.74,122.2) and (470.81,120.65) .. (469.2,119.04) .. controls (467.59,116.71) and (467.64,114.81) .. (469.33,113.34) .. controls (471.01,111.85) and (471.02,110.35) .. (469.35,108.86) .. controls (467.67,107.41) and (467.64,105.58) .. (469.26,103.37) .. controls (470.88,101.88) and (470.83,100.45) .. (469.1,99.09) .. controls (467.32,97.09) and (467.22,95.35) .. (468.79,93.87) .. controls (470.28,91.68) and (470.13,90) .. (468.34,88.81) .. controls (466.47,87.04) and (466.28,85.41) .. (467.77,83.92) .. controls (469.15,81.78) and (468.92,80.22) .. (467.08,79.23) .. controls (465.12,77.72) and (464.78,75.92) .. (466.07,73.85) .. controls (467.31,71.76) and (466.91,70.07) .. (464.88,68.79) .. controls (462.83,67.65) and (462.37,66.07) .. (463.49,64.06) .. controls (464.56,62.01) and (464.03,60.55) .. (461.92,59.68) .. controls (459.83,58.98) and (459.14,57.43) .. (459.85,55.04) .. controls (460.43,52.6) and (459.65,51.23) .. (457.52,50.94) -- (457.17,50.4) ;
\draw   (288,123.62) -- (326.1,123.62) -- (326.1,111.2) -- (346.4,129.68) -- (326.1,148.16) -- (326.1,135.74) -- (288,135.74) -- cycle ;

\draw (191.92,249.65) node [anchor=north west][inner sep=0.75pt]   [align=left] {v};
\draw (154.48,210.6) node [anchor=north west][inner sep=0.75pt]   [align=left] {$v_0$};
\draw (190.99,154.01) node [anchor=north west][inner sep=0.75pt]   [align=left] {$v_1$};
\draw (192.17,59.71) node [anchor=north west][inner sep=0.75pt]   [align=left] {$v_2$};
\draw (191.11,116.03) node [anchor=north west][inner sep=0.75pt]   [align=left] {$v_3$};
\draw (130.89,134.07) node [anchor=north west][inner sep=0.75pt]  [rotate=-2.23] [align=left] {$u_1$};
\draw (153.2,230.5) node [anchor=north west][inner sep=0.75pt]   [align=left] {u};
\draw (218.05,39.19) node [anchor=north west][inner sep=0.75pt]   [align=left] {$u_2$};
\draw (437.61,252.41) node [anchor=north west][inner sep=0.75pt]   [align=left] {v};
\draw (400.17,213.37) node [anchor=north west][inner sep=0.75pt]   [align=left] {$v_0$};
\draw (436.68,156.78) node [anchor=north west][inner sep=0.75pt]   [align=left] {$v_1$};
\draw (437.86,62.48) node [anchor=north west][inner sep=0.75pt]   [align=left] {$v_2$};
\draw (436.8,118.8) node [anchor=north west][inner sep=0.75pt]   [align=left] {$v_3$};
\draw (376.24,137.31) node [anchor=north west][inner sep=0.75pt]   [align=left] {$u_1$};
\draw (398.89,233.27) node [anchor=north west][inner sep=0.75pt]   [align=left] {u};
\draw (463.74,41.96) node [anchor=north west][inner sep=0.75pt]   [align=left] {$u_2$};
\draw (191.63,16.06) node [anchor=north west][inner sep=0.75pt]   [align=left] {$v_5$};
\draw (437.25,21.43) node [anchor=north west][inner sep=0.75pt]   [align=left] {$v_5$};

\end{tikzpicture}
    \caption{On the left side, we have the subgraph $ C $ after the first modification made in Case 1. On the right side, the second process mentioned in Case 1 was carried out. In red are the edges $ vu_3 $, $ u_3v_3 $, and $ v_3Tv_2 $, which are now part of the subgraph $ C $.
}
    \label{figfinal2}
\end{figure}
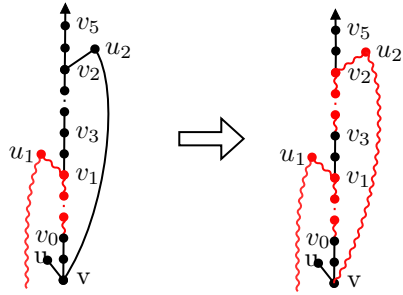

\textbf{Case 2:} $u\in L_j$.

Then $u$ has the unique successor in $T$, and thus it must be $w$ and $w\in L_j$. Let $v_1\in \lfloor w\rfloor\cap A_v$, it exists by Claim 3. Then we can find $u_1\in B\setminus B^{\infty}$ such that $v_1, v\in N(u_1)$. Now we remove edges $vu$ and $uw$ from $C$ and add the path $vu_1v_1Tv_0$ to $C$, where $v_0$ is the first vertex of $A$ above $w$ in $T$, see in Figure \ref{fig6}

\begin{figure}[ht]
    \centering

\tikzset{every picture/.style={line width=0.75pt}} 

\begin{tikzpicture}[x=0.4pt,y=0.4pt,yscale=-1,xscale=1]

\draw    (207.51,10) -- (207.59,41.73) ;
\draw [shift={(207.5,7)}, rotate = 89.85] [fill={rgb, 255:red, 0; green, 0; blue, 0 }  ][line width=0.08]  [draw opacity=0] (8.93,-4.29) -- (0,0) -- (8.93,4.29) -- cycle    ;
\draw    (207.59,41.73) -- (207.68,61.95) ;
\draw [shift={(207.59,41.73)}, rotate = 89.74] [color={rgb, 255:red, 0; green, 0; blue, 0 }  ][fill={rgb, 255:red, 0; green, 0; blue, 0 }  ][line width=0.75]      (0, 0) circle [x radius= 3.35, y radius= 3.35]   ;
\draw    (207.68,61.95) -- (207.77,82.18) ;
\draw [shift={(207.68,61.95)}, rotate = 89.74] [color={rgb, 255:red, 0; green, 0; blue, 0 }  ][fill={rgb, 255:red, 0; green, 0; blue, 0 }  ][line width=0.75]      (0, 0) circle [x radius= 3.35, y radius= 3.35]   ;
\draw    (207.77,82.18) -- (207.86,102.41) ;
\draw [shift={(207.77,82.18)}, rotate = 89.74] [color={rgb, 255:red, 0; green, 0; blue, 0 }  ][fill={rgb, 255:red, 0; green, 0; blue, 0 }  ][line width=0.75]      (0, 0) circle [x radius= 3.35, y radius= 3.35]   ;
\draw    (207.86,102.41) -- (207.95,122.64) ;
\draw [shift={(207.86,102.41)}, rotate = 89.74] [color={rgb, 255:red, 0; green, 0; blue, 0 }  ][fill={rgb, 255:red, 0; green, 0; blue, 0 }  ][line width=0.75]      (0, 0) circle [x radius= 3.35, y radius= 3.35]   ;
\draw    (207.95,122.64) -- (208.05,142.86) ;
\draw [shift={(207.95,122.64)}, rotate = 89.74] [color={rgb, 255:red, 0; green, 0; blue, 0 }  ][fill={rgb, 255:red, 0; green, 0; blue, 0 }  ][line width=0.75]      (0, 0) circle [x radius= 3.35, y radius= 3.35]   ;
\draw    (208.05,142.86) -- (208.14,163.09) ;
\draw [shift={(208.05,142.86)}, rotate = 89.74] [color={rgb, 255:red, 0; green, 0; blue, 0 }  ][fill={rgb, 255:red, 0; green, 0; blue, 0 }  ][line width=0.75]      (0, 0) circle [x radius= 3.35, y radius= 3.35]   ;
\draw [color={rgb, 255:red, 208; green, 2; blue, 27 }  ,draw opacity=1 ][fill={rgb, 255:red, 208; green, 2; blue, 27 }  ,fill opacity=1 ]   (208.14,163.09) .. controls (209.81,164.75) and (209.82,166.42) .. (208.16,168.09) .. controls (206.5,169.76) and (206.51,171.43) .. (208.18,173.09) .. controls (209.85,174.75) and (209.86,176.42) .. (208.2,178.09) .. controls (206.55,179.76) and (206.56,181.43) .. (208.23,183.09) -- (208.23,183.32) -- (208.23,183.32) ;
\draw [shift={(208.14,163.09)}, rotate = 89.74] [color={rgb, 255:red, 208; green, 2; blue, 27 }  ,draw opacity=1 ][fill={rgb, 255:red, 208; green, 2; blue, 27 }  ,fill opacity=1 ][line width=0.75]      (0, 0) circle [x radius= 3.35, y radius= 3.35]   ;
\draw [color={rgb, 255:red, 208; green, 2; blue, 27 }  ,draw opacity=1 ][fill={rgb, 255:red, 208; green, 2; blue, 27 }  ,fill opacity=1 ]   (208.23,183.32) .. controls (209.9,184.98) and (209.91,186.65) .. (208.25,188.32) .. controls (206.59,189.99) and (206.6,191.66) .. (208.27,193.32) .. controls (209.94,194.98) and (209.95,196.65) .. (208.29,198.32) .. controls (206.64,199.99) and (206.65,201.66) .. (208.32,203.32) -- (208.32,203.55) -- (208.32,203.55) ;
\draw [shift={(208.23,183.32)}, rotate = 89.74] [color={rgb, 255:red, 208; green, 2; blue, 27 }  ,draw opacity=1 ][fill={rgb, 255:red, 208; green, 2; blue, 27 }  ,fill opacity=1 ][line width=0.75]      (0, 0) circle [x radius= 3.35, y radius= 3.35]   ;
\draw    (208.32,203.55) -- (208.41,223.77) ;
\draw [shift={(208.32,203.55)}, rotate = 89.74] [color={rgb, 255:red, 0; green, 0; blue, 0 }  ][fill={rgb, 255:red, 0; green, 0; blue, 0 }  ][line width=0.75]      (0, 0) circle [x radius= 3.35, y radius= 3.35]   ;
\draw  [dash pattern={on 0.84pt off 2.51pt}]  (208.41,223.77) -- (208.5,254) ;
\draw [shift={(208.5,254)}, rotate = 89.83] [color={rgb, 255:red, 0; green, 0; blue, 0 }  ][fill={rgb, 255:red, 0; green, 0; blue, 0 }  ][line width=0.75]      (0, 0) circle [x radius= 3.35, y radius= 3.35]   ;
\draw [shift={(208.41,223.77)}, rotate = 89.83] [color={rgb, 255:red, 0; green, 0; blue, 0 }  ][fill={rgb, 255:red, 0; green, 0; blue, 0 }  ][line width=0.75]      (0, 0) circle [x radius= 3.35, y radius= 3.35]   ;
\draw    (187.5,22.5) .. controls (173,57) and (169.5,207) .. (208.32,203.55) ;
\draw    (187.5,22.5) -- (207.59,41.73) ;
\draw [shift={(187.5,22.5)}, rotate = 43.74] [color={rgb, 255:red, 0; green, 0; blue, 0 }  ][fill={rgb, 255:red, 0; green, 0; blue, 0 }  ][line width=0.75]      (0, 0) circle [x radius= 3.35, y radius= 3.35]   ;
\draw    (430.51,8.5) -- (430.59,40.23) ;
\draw [shift={(430.5,5.5)}, rotate = 89.85] [fill={rgb, 255:red, 0; green, 0; blue, 0 }  ][line width=0.08]  [draw opacity=0] (8.93,-4.29) -- (0,0) -- (8.93,4.29) -- cycle    ;
\draw [color={rgb, 255:red, 208; green, 2; blue, 27 }  ,draw opacity=1 ]   (430.59,40.23) .. controls (432.26,41.89) and (432.27,43.56) .. (430.61,45.23) .. controls (428.96,46.9) and (428.97,48.57) .. (430.64,50.23) .. controls (432.31,51.89) and (432.32,53.56) .. (430.66,55.23) .. controls (429,56.9) and (429.01,58.57) .. (430.68,60.23) -- (430.68,60.45) -- (430.68,60.45) ;
\draw [shift={(430.59,40.23)}, rotate = 89.74] [color={rgb, 255:red, 208; green, 2; blue, 27 }  ,draw opacity=1 ][fill={rgb, 255:red, 208; green, 2; blue, 27 }  ,fill opacity=1 ][line width=0.75]      (0, 0) circle [x radius= 3.35, y radius= 3.35]   ;
\draw [color={rgb, 255:red, 208; green, 2; blue, 27 }  ,draw opacity=1 ]   (430.68,60.45) .. controls (432.35,62.11) and (432.36,63.78) .. (430.7,65.45) .. controls (429.05,67.12) and (429.06,68.79) .. (430.73,70.45) .. controls (432.4,72.11) and (432.41,73.78) .. (430.75,75.45) .. controls (429.09,77.12) and (429.1,78.79) .. (430.77,80.45) -- (430.77,80.68) -- (430.77,80.68) ;
\draw [shift={(430.68,60.45)}, rotate = 89.74] [color={rgb, 255:red, 208; green, 2; blue, 27 }  ,draw opacity=1 ][fill={rgb, 255:red, 208; green, 2; blue, 27 }  ,fill opacity=1 ][line width=0.75]      (0, 0) circle [x radius= 3.35, y radius= 3.35]   ;
\draw [color={rgb, 255:red, 208; green, 2; blue, 27 }  ,draw opacity=1 ]   (430.77,80.68) .. controls (432.45,82.34) and (432.46,84.01) .. (430.8,85.68) .. controls (429.14,87.35) and (429.15,89.02) .. (430.82,90.68) .. controls (432.49,92.34) and (432.5,94.01) .. (430.84,95.68) .. controls (429.18,97.35) and (429.19,99.02) .. (430.86,100.68) -- (430.86,100.91) -- (430.86,100.91) ;
\draw [shift={(430.77,80.68)}, rotate = 89.74] [color={rgb, 255:red, 208; green, 2; blue, 27 }  ,draw opacity=1 ][fill={rgb, 255:red, 208; green, 2; blue, 27 }  ,fill opacity=1 ][line width=0.75]      (0, 0) circle [x radius= 3.35, y radius= 3.35]   ;
\draw [color={rgb, 255:red, 208; green, 2; blue, 27 }  ,draw opacity=1 ]   (430.86,100.91) .. controls (432.54,102.57) and (432.55,104.24) .. (430.89,105.91) .. controls (429.23,107.58) and (429.24,109.25) .. (430.91,110.91) .. controls (432.58,112.57) and (432.59,114.24) .. (430.93,115.91) .. controls (429.27,117.58) and (429.28,119.25) .. (430.95,120.91) -- (430.95,121.14) -- (430.95,121.14) ;
\draw [shift={(430.86,100.91)}, rotate = 89.74] [color={rgb, 255:red, 208; green, 2; blue, 27 }  ,draw opacity=1 ][fill={rgb, 255:red, 208; green, 2; blue, 27 }  ,fill opacity=1 ][line width=0.75]      (0, 0) circle [x radius= 3.35, y radius= 3.35]   ;
\draw    (430.95,121.14) -- (431.05,141.36) ;
\draw [shift={(430.95,121.14)}, rotate = 89.74] [color={rgb, 255:red, 0; green, 0; blue, 0 }  ][fill={rgb, 255:red, 0; green, 0; blue, 0 }  ][line width=0.75]      (0, 0) circle [x radius= 3.35, y radius= 3.35]   ;
\draw    (431.05,141.36) -- (431.14,161.59) ;
\draw [shift={(431.05,141.36)}, rotate = 89.74] [color={rgb, 255:red, 0; green, 0; blue, 0 }  ][fill={rgb, 255:red, 0; green, 0; blue, 0 }  ][line width=0.75]      (0, 0) circle [x radius= 3.35, y radius= 3.35]   ;
\draw [color={rgb, 255:red, 0; green, 0; blue, 0 }  ,draw opacity=1 ][fill={rgb, 255:red, 0; green, 0; blue, 0 }  ,fill opacity=1 ]   (431.14,161.59) -- (431.23,181.82) ;
\draw [shift={(431.14,161.59)}, rotate = 89.74] [color={rgb, 255:red, 0; green, 0; blue, 0 }  ,draw opacity=1 ][fill={rgb, 255:red, 0; green, 0; blue, 0 }  ,fill opacity=1 ][line width=0.75]      (0, 0) circle [x radius= 3.35, y radius= 3.35]   ;
\draw [color={rgb, 255:red, 0; green, 0; blue, 0 }  ,draw opacity=1 ][fill={rgb, 255:red, 0; green, 0; blue, 0 }  ,fill opacity=1 ]   (431.23,181.82) -- (431.32,202.05) ;
\draw [shift={(431.23,181.82)}, rotate = 89.74] [color={rgb, 255:red, 0; green, 0; blue, 0 }  ,draw opacity=1 ][fill={rgb, 255:red, 0; green, 0; blue, 0 }  ,fill opacity=1 ][line width=0.75]      (0, 0) circle [x radius= 3.35, y radius= 3.35]   ;
\draw    (431.32,202.05) -- (431.41,222.27) ;
\draw [shift={(431.32,202.05)}, rotate = 89.74] [color={rgb, 255:red, 0; green, 0; blue, 0 }  ][fill={rgb, 255:red, 0; green, 0; blue, 0 }  ][line width=0.75]      (0, 0) circle [x radius= 3.35, y radius= 3.35]   ;
\draw  [dash pattern={on 0.84pt off 2.51pt}]  (431.41,222.27) -- (431.5,252.5) ;
\draw [shift={(431.5,252.5)}, rotate = 89.83] [color={rgb, 255:red, 0; green, 0; blue, 0 }  ][fill={rgb, 255:red, 0; green, 0; blue, 0 }  ][line width=0.75]      (0, 0) circle [x radius= 3.35, y radius= 3.35]   ;
\draw [shift={(431.41,222.27)}, rotate = 89.83] [color={rgb, 255:red, 0; green, 0; blue, 0 }  ][fill={rgb, 255:red, 0; green, 0; blue, 0 }  ][line width=0.75]      (0, 0) circle [x radius= 3.35, y radius= 3.35]   ;
\draw [color={rgb, 255:red, 208; green, 2; blue, 27 }  ,draw opacity=1 ]   (410.5,21) .. controls (411.37,23.41) and (410.74,25.19) .. (408.61,26.33) .. controls (406.61,27.26) and (406.22,28.69) .. (407.44,30.6) .. controls (408.57,33.05) and (408.14,34.93) .. (406.16,36.22) .. controls (404.29,37.05) and (404.01,38.52) .. (405.31,40.63) .. controls (406.63,42.76) and (406.36,44.32) .. (404.51,45.33) .. controls (402.56,47.06) and (402.32,48.71) .. (403.78,50.3) .. controls (405.16,52.58) and (404.94,54.32) .. (403.11,55.52) .. controls (401.29,56.77) and (401.09,58.59) .. (402.51,60.97) .. controls (404.02,62.62) and (403.87,64.11) .. (402.07,65.46) .. controls (400.28,66.85) and (400.15,68.39) .. (401.68,70.07) .. controls (403.17,72.54) and (403.04,74.5) .. (401.27,75.97) .. controls (399.5,77.48) and (399.41,79.08) .. (400.99,80.79) .. controls (402.57,82.49) and (402.49,84.12) .. (400.76,85.67) .. controls (399.03,87.26) and (398.97,88.91) .. (400.58,90.62) .. controls (402.2,92.32) and (402.16,93.98) .. (400.45,95.61) .. controls (398.76,97.26) and (398.73,98.93) .. (400.38,100.63) .. controls (402.04,102.31) and (402.04,103.99) .. (400.37,105.67) .. controls (398.72,107.36) and (398.73,109.04) .. (400.42,110.71) .. controls (402.11,112.36) and (402.15,114.04) .. (400.52,115.75) .. controls (398.91,117.48) and (398.96,119.16) .. (400.69,120.77) .. controls (402.42,122.36) and (402.5,124.02) .. (400.91,125.76) .. controls (399.34,127.51) and (399.44,129.15) .. (401.21,130.7) .. controls (402.98,132.21) and (403.1,133.83) .. (401.56,135.58) .. controls (400.04,137.33) and (400.18,138.94) .. (401.99,140.39) .. controls (403.8,141.8) and (403.97,143.38) .. (402.48,145.12) .. controls (401.01,146.87) and (401.25,148.8) .. (403.2,150.89) .. controls (405.06,152.16) and (405.28,153.65) .. (403.85,155.38) .. controls (402.44,157.11) and (402.75,158.92) .. (404.77,160.81) .. controls (406.67,161.9) and (406.94,163.29) .. (405.59,165) .. controls (404.41,167.4) and (404.79,169.07) .. (406.72,170.01) .. controls (408.82,171.5) and (409.24,173.09) .. (407.98,174.76) .. controls (406.93,177.07) and (407.39,178.55) .. (409.36,179.21) .. controls (411.51,180.31) and (412.12,181.95) .. (411.19,184.13) .. controls (410.33,186.34) and (411,187.82) .. (413.2,188.56) .. controls (415.61,189.47) and (416.48,190.95) .. (415.81,193.02) .. controls (415.58,195.53) and (416.69,196.89) .. (419.12,197.12) .. controls (421.35,196.91) and (422.58,197.89) .. (422.8,200.07) .. controls (423.49,202.28) and (425.02,202.87) .. (427.4,201.86) -- (431.32,202.05) ;
\draw [color={rgb, 255:red, 208; green, 2; blue, 27 }  ,draw opacity=1 ]   (410.5,21) .. controls (412.85,20.95) and (414.06,22.11) .. (414.11,24.46) .. controls (414.16,26.81) and (415.37,27.96) .. (417.72,27.91) .. controls (420.08,27.86) and (421.29,29.01) .. (421.34,31.37) .. controls (421.39,33.72) and (422.6,34.88) .. (424.95,34.83) .. controls (427.3,34.78) and (428.51,35.94) .. (428.56,38.29) -- (430.59,40.23) -- (430.59,40.23) ;
\draw [shift={(410.5,21)}, rotate = 43.74] [color={rgb, 255:red, 208; green, 2; blue, 27 }  ,draw opacity=1 ][fill={rgb, 255:red, 208; green, 2; blue, 27 }  ,fill opacity=1 ][line width=0.75]      (0, 0) circle [x radius= 3.35, y radius= 3.35]   ;
\draw   (298,104.75) -- (329.8,104.75) -- (329.8,97.5) -- (351,112) -- (329.8,126.5) -- (329.8,119.25) -- (298,119.25) -- cycle ;

\draw (215.4,194.3) node [anchor=north west][inner sep=0.75pt]    {$v$};
\draw (216,174.6) node [anchor=north west][inner sep=0.75pt]    {$u$};
\draw (215,155) node [anchor=north west][inner sep=0.75pt]    {$w$};
\draw (215.5,33) node [anchor=north west][inner sep=0.75pt]    {$v_{1} \in A_{v}$};
\draw (165.5,4.9) node [anchor=north west][inner sep=0.75pt]    {$u_{1}$};
\draw (216.5,240.6) node [anchor=north west][inner sep=0.75pt]    {$v_{j}^{n_{0}}$};
\draw (438,193.6) node [anchor=north west][inner sep=0.75pt]    {$v$};
\draw (439,173.5) node [anchor=north west][inner sep=0.75pt]    {$u$};
\draw (438,153.9) node [anchor=north west][inner sep=0.75pt]    {$w$};
\draw (438.5,32.3) node [anchor=north west][inner sep=0.75pt]    {$v_{1} \in A_{v}$};
\draw (388.5,3.4) node [anchor=north west][inner sep=0.75pt]    {$u_{1}$};
\draw (439.5,239.5) node [anchor=north west][inner sep=0.75pt]    {$v_{j}^{n_{0}}$};
\draw (437.5,112.9) node [anchor=north west][inner sep=0.75pt]    {$v_{0}$};

\end{tikzpicture}

    \caption{First part of the process of adding and removing paths from cycle $C$ - Case 2.}
    \label{fig6}
\end{figure}
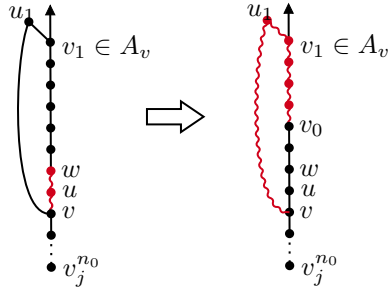

Now let $v_2\in \lfloor v_1\rfloor\cap A_w$, it exists by Claim 3. Then there exists $u_2\in B\setminus B^{\infty}$ such that $v_2, w\in N(u_2)$. Let us add the path $wu_2v_2Tv_3$ into $C$, where $v_3\in A$ is the minimal one above $v_1$ in the tree order, see Figure \ref{fig7}. Repeating the process for all $v_i$'s obtained, we get a double ray $R_j$ covering $(T^{<n_0}\cup L_j)\cap A$ such that each $a\in (A\cap C)\setminus L_j$ has the same two neighbors in $C$ and in $R_j$. This finishes Case 2.

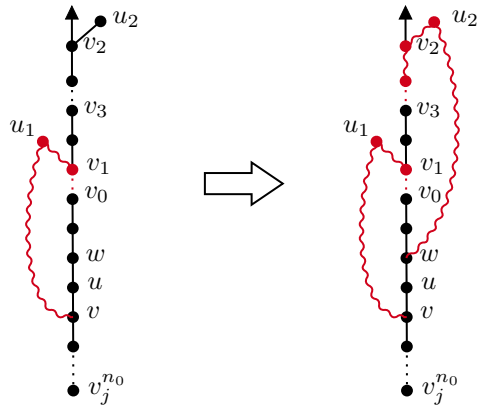
\begin{figure}[ht]
    \centering

\tikzset{every picture/.style={line width=0.75pt}} 

\begin{tikzpicture}[x=0.55pt,y=0.55pt,yscale=-1,xscale=1]

\draw    (207.5,5) -- (207.55,29.51) ;
\draw [shift={(207.5,2)}, rotate = 89.91] [fill={rgb, 255:red, 0; green, 0; blue, 0 }  ][line width=0.08]  [draw opacity=0] (8.93,-4.29) -- (0,0) -- (8.93,4.29) -- cycle    ;
\draw    (207.55,29.51) -- (207.59,53.53) ;
\draw [shift={(207.55,29.51)}, rotate = 89.89] [color={rgb, 255:red, 0; green, 0; blue, 0 }  ][fill={rgb, 255:red, 0; green, 0; blue, 0 }  ][line width=0.75]      (0, 0) circle [x radius= 3.35, y radius= 3.35]   ;
\draw  [dash pattern={on 0.84pt off 2.51pt}]  (207.59,53.53) -- (207.68,73.75) ;
\draw [shift={(207.59,53.53)}, rotate = 89.74] [color={rgb, 255:red, 0; green, 0; blue, 0 }  ][fill={rgb, 255:red, 0; green, 0; blue, 0 }  ][line width=0.75]      (0, 0) circle [x radius= 3.35, y radius= 3.35]   ;
\draw    (207.68,73.75) -- (207.77,93.98) ;
\draw [shift={(207.68,73.75)}, rotate = 89.74] [color={rgb, 255:red, 0; green, 0; blue, 0 }  ][fill={rgb, 255:red, 0; green, 0; blue, 0 }  ][line width=0.75]      (0, 0) circle [x radius= 3.35, y radius= 3.35]   ;
\draw    (207.77,93.98) -- (207.86,114.21) ;
\draw [shift={(207.77,93.98)}, rotate = 89.74] [color={rgb, 255:red, 0; green, 0; blue, 0 }  ][fill={rgb, 255:red, 0; green, 0; blue, 0 }  ][line width=0.75]      (0, 0) circle [x radius= 3.35, y radius= 3.35]   ;
\draw [color={rgb, 255:red, 208; green, 2; blue, 27 }  ,draw opacity=1 ] [dash pattern={on 0.84pt off 2.51pt}]  (207.86,114.21) -- (207.95,134.44) ;
\draw [shift={(207.86,114.21)}, rotate = 89.74] [color={rgb, 255:red, 208; green, 2; blue, 27 }  ,draw opacity=1 ][fill={rgb, 255:red, 208; green, 2; blue, 27 }  ,fill opacity=1 ][line width=0.75]      (0, 0) circle [x radius= 3.35, y radius= 3.35]   ;
\draw    (207.95,134.44) -- (208.05,154.66) ;
\draw [shift={(207.95,134.44)}, rotate = 89.74] [color={rgb, 255:red, 0; green, 0; blue, 0 }  ][fill={rgb, 255:red, 0; green, 0; blue, 0 }  ][line width=0.75]      (0, 0) circle [x radius= 3.35, y radius= 3.35]   ;
\draw    (208.05,154.66) -- (208.14,174.89) ;
\draw [shift={(208.05,154.66)}, rotate = 89.74] [color={rgb, 255:red, 0; green, 0; blue, 0 }  ][fill={rgb, 255:red, 0; green, 0; blue, 0 }  ][line width=0.75]      (0, 0) circle [x radius= 3.35, y radius= 3.35]   ;
\draw [color={rgb, 255:red, 0; green, 0; blue, 0 }  ,draw opacity=1 ][fill={rgb, 255:red, 0; green, 0; blue, 0 }  ,fill opacity=1 ]   (208.14,174.89) -- (208.23,195.12) ;
\draw [shift={(208.14,174.89)}, rotate = 89.74] [color={rgb, 255:red, 0; green, 0; blue, 0 }  ,draw opacity=1 ][fill={rgb, 255:red, 0; green, 0; blue, 0 }  ,fill opacity=1 ][line width=0.75]      (0, 0) circle [x radius= 3.35, y radius= 3.35]   ;
\draw [color={rgb, 255:red, 0; green, 0; blue, 0 }  ,draw opacity=1 ][fill={rgb, 255:red, 0; green, 0; blue, 0 }  ,fill opacity=1 ]   (208.23,195.12) -- (208.32,215.35) ;
\draw [shift={(208.23,195.12)}, rotate = 89.74] [color={rgb, 255:red, 0; green, 0; blue, 0 }  ,draw opacity=1 ][fill={rgb, 255:red, 0; green, 0; blue, 0 }  ,fill opacity=1 ][line width=0.75]      (0, 0) circle [x radius= 3.35, y radius= 3.35]   ;
\draw    (208.32,215.35) -- (208.41,235.57) ;
\draw [shift={(208.32,215.35)}, rotate = 89.74] [color={rgb, 255:red, 0; green, 0; blue, 0 }  ][fill={rgb, 255:red, 0; green, 0; blue, 0 }  ][line width=0.75]      (0, 0) circle [x radius= 3.35, y radius= 3.35]   ;
\draw  [dash pattern={on 0.84pt off 2.51pt}]  (208.41,235.57) -- (208.5,265.8) ;
\draw [shift={(208.5,265.8)}, rotate = 89.83] [color={rgb, 255:red, 0; green, 0; blue, 0 }  ][fill={rgb, 255:red, 0; green, 0; blue, 0 }  ][line width=0.75]      (0, 0) circle [x radius= 3.35, y radius= 3.35]   ;
\draw [shift={(208.41,235.57)}, rotate = 89.83] [color={rgb, 255:red, 0; green, 0; blue, 0 }  ][fill={rgb, 255:red, 0; green, 0; blue, 0 }  ][line width=0.75]      (0, 0) circle [x radius= 3.35, y radius= 3.35]   ;
\draw [color={rgb, 255:red, 208; green, 2; blue, 27 }  ,draw opacity=1 ]   (187.77,94.98) .. controls (188.62,97.36) and (187.99,99.02) .. (185.88,99.96) .. controls (183.77,101.06) and (183.25,102.7) .. (184.33,104.88) .. controls (185.57,106.62) and (185.15,108.16) .. (183.07,109.51) .. controls (181,110.99) and (180.62,112.63) .. (181.92,114.43) .. controls (183.25,116.21) and (182.9,117.92) .. (180.89,119.57) .. controls (178.98,120.8) and (178.71,122.32) .. (180.1,124.14) .. controls (181.51,125.94) and (181.24,127.77) .. (179.3,129.62) .. controls (177.45,131.01) and (177.26,132.61) .. (178.73,134.42) .. controls (180.22,136.2) and (180.07,137.82) .. (178.26,139.29) .. controls (176.47,140.8) and (176.35,142.44) .. (177.9,144.2) .. controls (179.48,145.93) and (179.4,147.57) .. (177.67,149.14) .. controls (175.95,150.75) and (175.91,152.39) .. (177.55,154.07) .. controls (179.21,155.7) and (179.21,157.34) .. (177.56,158.97) .. controls (175.93,160.64) and (175.98,162.26) .. (177.71,163.83) .. controls (179.49,165.86) and (179.6,167.71) .. (178.05,169.4) .. controls (176.53,171.12) and (176.68,172.67) .. (178.5,174.06) .. controls (180.34,175.37) and (180.58,177.13) .. (179.22,179.34) .. controls (177.82,181.1) and (178.08,182.55) .. (180,183.7) .. controls (182.03,185.19) and (182.4,186.81) .. (181.11,188.55) .. controls (180,190.8) and (180.52,192.53) .. (182.66,193.76) .. controls (184.81,194.79) and (185.42,196.38) .. (184.51,198.54) .. controls (183.72,200.77) and (184.45,202.21) .. (186.69,202.85) .. controls (188.89,203.23) and (189.84,204.63) .. (189.54,207.05) .. controls (189.26,209.27) and (190.36,210.41) .. (192.83,210.48) .. controls (195.11,210.14) and (196.5,211.08) .. (197.01,213.31) .. controls (198,215.54) and (199.59,216.11) .. (201.78,215) .. controls (203.34,213.56) and (204.95,213.71) .. (206.6,215.44) -- (208.32,215.35) ;
\draw [color={rgb, 255:red, 208; green, 2; blue, 27 }  ,draw opacity=1 ]   (187.77,94.98) .. controls (190.13,94.93) and (191.34,96.08) .. (191.39,98.44) .. controls (191.44,100.79) and (192.65,101.95) .. (195,101.9) .. controls (197.35,101.85) and (198.56,103) .. (198.61,105.35) .. controls (198.66,107.7) and (199.87,108.86) .. (202.22,108.81) .. controls (204.57,108.76) and (205.78,109.92) .. (205.83,112.27) -- (207.86,114.21) -- (207.86,114.21) ;
\draw [shift={(187.77,94.98)}, rotate = 43.74] [color={rgb, 255:red, 208; green, 2; blue, 27 }  ,draw opacity=1 ][fill={rgb, 255:red, 208; green, 2; blue, 27 }  ,fill opacity=1 ][line width=0.75]      (0, 0) circle [x radius= 3.35, y radius= 3.35]   ;
\draw   (298,116.55) -- (329.8,116.55) -- (329.8,109.3) -- (351,123.8) -- (329.8,138.3) -- (329.8,131.05) -- (298,131.05) -- cycle ;
\draw    (227,12.5) -- (207.55,29.51) ;
\draw [shift={(227,12.5)}, rotate = 138.83] [color={rgb, 255:red, 0; green, 0; blue, 0 }  ][fill={rgb, 255:red, 0; green, 0; blue, 0 }  ][line width=0.75]      (0, 0) circle [x radius= 3.35, y radius= 3.35]   ;
\draw    (434.5,5) -- (434.55,29.51) ;
\draw [shift={(434.5,2)}, rotate = 89.91] [fill={rgb, 255:red, 0; green, 0; blue, 0 }  ][line width=0.08]  [draw opacity=0] (8.93,-4.29) -- (0,0) -- (8.93,4.29) -- cycle    ;
\draw [color={rgb, 255:red, 208; green, 2; blue, 27 }  ,draw opacity=1 ]   (434.55,29.51) .. controls (436.22,31.18) and (436.22,32.84) .. (434.55,34.51) .. controls (432.89,36.18) and (432.89,37.84) .. (434.56,39.51) .. controls (436.23,41.18) and (436.23,42.84) .. (434.57,44.51) .. controls (432.91,46.18) and (432.91,47.84) .. (434.58,49.51) -- (434.59,53.53) -- (434.59,53.53) ;
\draw [shift={(434.55,29.51)}, rotate = 89.89] [color={rgb, 255:red, 208; green, 2; blue, 27 }  ,draw opacity=1 ][fill={rgb, 255:red, 208; green, 2; blue, 27 }  ,fill opacity=1 ][line width=0.75]      (0, 0) circle [x radius= 3.35, y radius= 3.35]   ;
\draw [color={rgb, 255:red, 208; green, 2; blue, 27 }  ,draw opacity=1 ] [dash pattern={on 0.84pt off 2.51pt}]  (434.59,53.53) -- (434.68,73.75) ;
\draw [shift={(434.59,53.53)}, rotate = 89.74] [color={rgb, 255:red, 208; green, 2; blue, 27 }  ,draw opacity=1 ][fill={rgb, 255:red, 208; green, 2; blue, 27 }  ,fill opacity=1 ][line width=0.75]      (0, 0) circle [x radius= 3.35, y radius= 3.35]   ;
\draw    (434.68,73.75) -- (434.77,93.98) ;
\draw [shift={(434.68,73.75)}, rotate = 89.74] [color={rgb, 255:red, 0; green, 0; blue, 0 }  ][fill={rgb, 255:red, 0; green, 0; blue, 0 }  ][line width=0.75]      (0, 0) circle [x radius= 3.35, y radius= 3.35]   ;
\draw    (434.77,93.98) -- (434.86,114.21) ;
\draw [shift={(434.77,93.98)}, rotate = 89.74] [color={rgb, 255:red, 0; green, 0; blue, 0 }  ][fill={rgb, 255:red, 0; green, 0; blue, 0 }  ][line width=0.75]      (0, 0) circle [x radius= 3.35, y radius= 3.35]   ;
\draw [color={rgb, 255:red, 208; green, 2; blue, 27 }  ,draw opacity=1 ] [dash pattern={on 0.84pt off 2.51pt}]  (434.86,114.21) -- (434.95,134.44) ;
\draw [shift={(434.86,114.21)}, rotate = 89.74] [color={rgb, 255:red, 208; green, 2; blue, 27 }  ,draw opacity=1 ][fill={rgb, 255:red, 208; green, 2; blue, 27 }  ,fill opacity=1 ][line width=0.75]      (0, 0) circle [x radius= 3.35, y radius= 3.35]   ;
\draw    (434.95,134.44) -- (435.05,154.66) ;
\draw [shift={(434.95,134.44)}, rotate = 89.74] [color={rgb, 255:red, 0; green, 0; blue, 0 }  ][fill={rgb, 255:red, 0; green, 0; blue, 0 }  ][line width=0.75]      (0, 0) circle [x radius= 3.35, y radius= 3.35]   ;
\draw    (435.05,154.66) -- (435.14,174.89) ;
\draw [shift={(435.05,154.66)}, rotate = 89.74] [color={rgb, 255:red, 0; green, 0; blue, 0 }  ][fill={rgb, 255:red, 0; green, 0; blue, 0 }  ][line width=0.75]      (0, 0) circle [x radius= 3.35, y radius= 3.35]   ;
\draw [color={rgb, 255:red, 0; green, 0; blue, 0 }  ,draw opacity=1 ][fill={rgb, 255:red, 0; green, 0; blue, 0 }  ,fill opacity=1 ]   (435.14,174.89) -- (435.23,195.12) ;
\draw [shift={(435.14,174.89)}, rotate = 89.74] [color={rgb, 255:red, 0; green, 0; blue, 0 }  ,draw opacity=1 ][fill={rgb, 255:red, 0; green, 0; blue, 0 }  ,fill opacity=1 ][line width=0.75]      (0, 0) circle [x radius= 3.35, y radius= 3.35]   ;
\draw [color={rgb, 255:red, 0; green, 0; blue, 0 }  ,draw opacity=1 ][fill={rgb, 255:red, 0; green, 0; blue, 0 }  ,fill opacity=1 ]   (435.23,195.12) -- (435.32,215.35) ;
\draw [shift={(435.23,195.12)}, rotate = 89.74] [color={rgb, 255:red, 0; green, 0; blue, 0 }  ,draw opacity=1 ][fill={rgb, 255:red, 0; green, 0; blue, 0 }  ,fill opacity=1 ][line width=0.75]      (0, 0) circle [x radius= 3.35, y radius= 3.35]   ;
\draw    (435.32,215.35) -- (435.41,235.57) ;
\draw [shift={(435.32,215.35)}, rotate = 89.74] [color={rgb, 255:red, 0; green, 0; blue, 0 }  ][fill={rgb, 255:red, 0; green, 0; blue, 0 }  ][line width=0.75]      (0, 0) circle [x radius= 3.35, y radius= 3.35]   ;
\draw  [dash pattern={on 0.84pt off 2.51pt}]  (435.41,235.57) -- (435.5,265.8) ;
\draw [shift={(435.5,265.8)}, rotate = 89.83] [color={rgb, 255:red, 0; green, 0; blue, 0 }  ][fill={rgb, 255:red, 0; green, 0; blue, 0 }  ][line width=0.75]      (0, 0) circle [x radius= 3.35, y radius= 3.35]   ;
\draw [shift={(435.41,235.57)}, rotate = 89.83] [color={rgb, 255:red, 0; green, 0; blue, 0 }  ][fill={rgb, 255:red, 0; green, 0; blue, 0 }  ][line width=0.75]      (0, 0) circle [x radius= 3.35, y radius= 3.35]   ;
\draw [color={rgb, 255:red, 208; green, 2; blue, 27 }  ,draw opacity=1 ]   (414.77,94.98) .. controls (415.62,97.36) and (414.99,99.02) .. (412.88,99.96) .. controls (410.77,101.06) and (410.25,102.7) .. (411.33,104.88) .. controls (412.57,106.62) and (412.15,108.16) .. (410.07,109.51) .. controls (408,110.99) and (407.62,112.63) .. (408.92,114.43) .. controls (410.25,116.21) and (409.9,117.92) .. (407.89,119.57) .. controls (405.98,120.8) and (405.71,122.32) .. (407.1,124.14) .. controls (408.51,125.94) and (408.24,127.77) .. (406.3,129.62) .. controls (404.45,131.01) and (404.26,132.61) .. (405.73,134.42) .. controls (407.22,136.2) and (407.07,137.82) .. (405.26,139.29) .. controls (403.47,140.8) and (403.35,142.44) .. (404.9,144.2) .. controls (406.48,145.93) and (406.4,147.57) .. (404.67,149.14) .. controls (402.95,150.75) and (402.91,152.39) .. (404.55,154.07) .. controls (406.21,155.7) and (406.21,157.34) .. (404.56,158.97) .. controls (402.93,160.64) and (402.98,162.26) .. (404.71,163.83) .. controls (406.49,165.86) and (406.6,167.71) .. (405.05,169.4) .. controls (403.53,171.12) and (403.68,172.67) .. (405.5,174.06) .. controls (407.34,175.37) and (407.58,177.13) .. (406.22,179.34) .. controls (404.82,181.1) and (405.08,182.55) .. (407,183.7) .. controls (409.03,185.19) and (409.4,186.81) .. (408.11,188.55) .. controls (407,190.8) and (407.52,192.53) .. (409.66,193.76) .. controls (411.81,194.79) and (412.42,196.38) .. (411.51,198.54) .. controls (410.72,200.77) and (411.45,202.21) .. (413.69,202.85) .. controls (415.89,203.23) and (416.84,204.63) .. (416.54,207.05) .. controls (416.26,209.27) and (417.36,210.41) .. (419.83,210.48) .. controls (422.11,210.14) and (423.5,211.08) .. (424.01,213.31) .. controls (425,215.54) and (426.59,216.11) .. (428.78,215) .. controls (430.34,213.56) and (431.95,213.71) .. (433.6,215.44) -- (435.32,215.35) ;
\draw [color={rgb, 255:red, 208; green, 2; blue, 27 }  ,draw opacity=1 ]   (414.77,94.98) .. controls (417.13,94.93) and (418.34,96.08) .. (418.39,98.44) .. controls (418.44,100.79) and (419.65,101.95) .. (422,101.9) .. controls (424.35,101.85) and (425.56,103) .. (425.61,105.35) .. controls (425.66,107.7) and (426.87,108.86) .. (429.22,108.81) .. controls (431.57,108.76) and (432.78,109.92) .. (432.83,112.27) -- (434.86,114.21) -- (434.86,114.21) ;
\draw [shift={(414.77,94.98)}, rotate = 43.74] [color={rgb, 255:red, 208; green, 2; blue, 27 }  ,draw opacity=1 ][fill={rgb, 255:red, 208; green, 2; blue, 27 }  ,fill opacity=1 ][line width=0.75]      (0, 0) circle [x radius= 3.35, y radius= 3.35]   ;
\draw [color={rgb, 255:red, 208; green, 2; blue, 27 }  ,draw opacity=1 ]   (454,13) .. controls (453.81,15.35) and (452.54,16.43) .. (450.19,16.24) .. controls (447.84,16.05) and (446.57,17.12) .. (446.38,19.47) .. controls (446.19,21.82) and (444.91,22.9) .. (442.56,22.71) .. controls (440.21,22.52) and (438.94,23.59) .. (438.75,25.94) .. controls (438.56,28.29) and (437.29,29.37) .. (434.94,29.18) -- (434.55,29.51) -- (434.55,29.51) ;
\draw [shift={(454,13)}, rotate = 139.67] [color={rgb, 255:red, 208; green, 2; blue, 27 }  ,draw opacity=1 ][fill={rgb, 255:red, 208; green, 2; blue, 27 }  ,fill opacity=1 ][line width=0.75]      (0, 0) circle [x radius= 3.35, y radius= 3.35]   ;
\draw [color={rgb, 255:red, 208; green, 2; blue, 27 }  ,draw opacity=1 ]   (435.14,174.89) .. controls (435.54,172.52) and (436.94,171.43) .. (439.33,171.63) .. controls (441.78,171.69) and (443.07,170.5) .. (443.21,168.05) .. controls (443.2,165.66) and (444.39,164.37) .. (446.78,164.18) .. controls (448.83,164.33) and (449.74,163.18) .. (449.53,160.74) .. controls (449.2,158.35) and (450.21,156.9) .. (452.56,156.38) .. controls (454.93,155.72) and (455.7,154.45) .. (454.87,152.56) .. controls (454.26,150.21) and (455.11,148.61) .. (457.4,147.78) .. controls (459.69,146.82) and (460.33,145.44) .. (459.3,143.64) .. controls (458.45,141.33) and (459.13,139.62) .. (461.34,138.51) .. controls (463.54,137.3) and (464.04,135.83) .. (462.85,134.11) .. controls (461.8,131.85) and (462.34,130.05) .. (464.45,128.7) .. controls (466.38,127.89) and (466.76,126.36) .. (465.59,124.1) .. controls (464.37,121.9) and (464.7,120.34) .. (466.58,119.43) .. controls (468.57,117.82) and (468.85,116.25) .. (467.42,114.7) .. controls (466.06,112.55) and (466.33,110.64) .. (468.23,108.95) .. controls (470.04,107.85) and (470.21,106.24) .. (468.75,104.13) .. controls (467.2,102.7) and (467.33,101.09) .. (469.14,99.28) .. controls (470.88,98.08) and (470.96,96.46) .. (469.39,94.43) .. controls (467.78,92.45) and (467.83,90.83) .. (469.52,89.58) .. controls (471.19,87.64) and (471.19,86.03) .. (469.52,84.74) .. controls (467.82,82.85) and (467.78,81.25) .. (469.41,79.93) .. controls (470.99,77.93) and (470.9,76.02) .. (469.13,74.2) .. controls (467.38,73.06) and (467.27,71.49) .. (468.78,69.49) .. controls (470.26,67.46) and (470.11,65.92) .. (468.33,64.85) .. controls (466.46,63.22) and (466.24,61.39) .. (467.66,59.38) .. controls (469.05,57.34) and (468.83,55.85) .. (467,54.92) .. controls (465.07,53.45) and (464.78,51.72) .. (466.11,49.71) .. controls (467.42,47.68) and (467.08,46) .. (465.1,44.67) .. controls (463.11,43.41) and (462.74,41.79) .. (463.99,39.82) .. controls (465.2,37.83) and (464.8,36.29) .. (462.78,35.18) .. controls (460.75,34.16) and (460.25,32.46) .. (461.26,30.07) .. controls (462.38,28.14) and (461.92,26.77) .. (459.87,25.96) .. controls (457.68,24.85) and (457.03,23.17) .. (457.92,20.93) .. controls (458.76,18.69) and (458.07,17.19) .. (455.86,16.44) -- (454,13) ;

\draw (215.4,206.1) node [anchor=north west][inner sep=0.75pt]    {$v$};
\draw (216,186.4) node [anchor=north west][inner sep=0.75pt]    {$u$};
\draw (215,166.8) node [anchor=north west][inner sep=0.75pt]    {$w$};
\draw (215.5,104.8) node [anchor=north west][inner sep=0.75pt]    {$v_{1}$};
\draw (163,78.7) node [anchor=north west][inner sep=0.75pt]    {$u_{1}$};
\draw (216.5,252.4) node [anchor=north west][inner sep=0.75pt]    {$v_{j}^{n_{0}}$};
\draw (214,123.7) node [anchor=north west][inner sep=0.75pt]    {$v_{0}$};
\draw (233,3.9) node [anchor=north west][inner sep=0.75pt]    {$u_{2}$};
\draw (442.4,206.1) node [anchor=north west][inner sep=0.75pt]    {$v$};
\draw (443,186.4) node [anchor=north west][inner sep=0.75pt]    {$u$};
\draw (442,166.8) node [anchor=north west][inner sep=0.75pt]    {$w$};
\draw (442.5,104.8) node [anchor=north west][inner sep=0.75pt]    {$v_{1}$};
\draw (390,78.7) node [anchor=north west][inner sep=0.75pt]    {$u_{1}$};
\draw (443.5,252.4) node [anchor=north west][inner sep=0.75pt]    {$v_{j}^{n_{0}}$};
\draw (441,123.7) node [anchor=north west][inner sep=0.75pt]    {$v_{0}$};
\draw (464,2.9) node [anchor=north west][inner sep=0.75pt]    {$u_{2}$};
\draw (213.5,65.4) node [anchor=north west][inner sep=0.75pt]    {$v_{3}$};
\draw (440,65.9) node [anchor=north west][inner sep=0.75pt]    {$v_{3}$};
\draw (213,21.9) node [anchor=north west][inner sep=0.75pt]    {$v_{2}$};
\draw (439.5,22.4) node [anchor=north west][inner sep=0.75pt]    {$v_{2}$};

\end{tikzpicture}
    \caption{Second part of the process of adding and removing paths from cycle $C$ - Case 2.}
    \label{fig7}
\end{figure}

Applying this process, we can extend the cycle $C$ to a family of double rays. For each infinite branch $L_j$, we construct two disjoint rays that are connected to the initial cycle $C$. Since these two rays are equivalent (they are infinitely connected through the same branch), the family of double rays we obtain forms an infinite cycle in $\vert G\vert$. Since this infinite cycle covers all the vertices of $A$, it is the desired cycle.

\end{proof}

\section*{Acknowledgments}
The first named author thanks the support of Fundação de Amparo à Pesquisa do Estado de São Paulo (FAPESP), being sponsored through grant number 2023/00595-6. The second named author acknowledges
the support of Conselho Nacional de Desenvolvimento Científico e Tecnológico (CNPq) through grant number 165761/2021-0. The third named author acknowledges the support of the Austrian Science Fund (FWF) through grant number [10.55776/I5930 and 10.55776/PAT5730424].

\Addresses

\end{document}